\documentclass[a4paper,10pt]{article}
\usepackage[utf8]{inputenc}
\usepackage{amsmath}
\usepackage{amssymb}
\usepackage{amsthm}

\theoremstyle{plain}
\newtheorem{thm}{Theorem}[section]
\newtheorem{Prop}[thm]{Proposition}
\newtheorem{Def}[thm]{Definition}
\newtheorem{lma}[thm]{Lemma}

\newtheorem{ex}{Example}

\newtheorem*{example*}{Example}
\newtheorem*{remark*}{Remark}

\newcommand{\Ch}{\mathrm{Ch}}

\newcommand{\Ric}{\mathrm{Ric}}

\newcommand{\Dom}{\mathit{Dom}}
\newcommand{\Lip}{\mathrm{Lip}}
\newcommand{\F}{\mathcal{F}}
\newcommand{\E}{\mathcal{E}}
\newcommand{\lip}{\mathrm{lip}}

\title{Super-Ricci flows and improved gradient and transport estimates}
\author{Eva Kopfer\thanks{Institut f\"ur Angewandte Mathematik, Universit\"at Bonn, Endenicher Allee 60, 53115 Bonn, Germany (\texttt{eva.kopfer@iam.uni-bonn.de})}}
\date{}

\begin{document}

\maketitle

\begin{abstract}
We introduce Brownian motions on time-dependent metric measure spaces, proving their existence and uniqueness. We prove contraction estimates for their trajectories assuming that the time-dependent heat flow satisfies transport estimates with respect to every
$L^p$-Kantorovich distance, $p\in[1,\infty]$.
These transport estimates turn out to characterize super-Ricci flows, introduced by Sturm in \cite{sturm2015}.
\end{abstract}


\section{Introduction and statement of the main results}
The heat flow on a metric measure space $(X,d,m)$ can be understood either as the gradient flow of the Cheeger energy $\Ch$ on the Hilbert space $L^2(X,m)$ or as the gradient flow of the relative entropy $S$ on the space of probability measures $\mathcal P_2(X)$ endowed with the $L^2$-Kantorovich distance $W_2$. It has been shown in \cite{agscalc} that these two notions coincide under the assumption that $(X,d,m)$ satisfies a lower Ricci curvature bound in the sense of Lott-Sturm-Villani. 

A metric measure space is said to have a lower Ricci curvature bound $K$, in short CD$(K,\infty)$, if the relative entropy is $K$-convex along $L^2$-Kantorovich geodesics. This definition is consistent with the Riemannian case, i.e. a Riemannian manifold has Ricci curvature bounded from below by $K$ if and only if its relative entropy is $K$-convex. But then even more holds true. Each of the following properties characterize a lower curvature bounds of the manifold: $L^p$-transport and gradient estimates of the heat flow for $p\in[1,\infty]$, and pathwise contraction for Brownian trajectories, see e.g. \cite{sturmrenesse}.

On general CD$(K,\infty)$-spaces these properties fail, and the heat flow does not even have to be linear. In order to obtain a more Riemannian-like behavior, Ambrosio, Gigli and Savar\'e introduced in \cite{agsmet} the notion of RCD$(K,\infty)$-spaces, i.e. CD$(K,\infty)$-spaces whose heat flow is linear. This notion is characterized by one single formula, namely the \textit{Evolution variational inequality} of the heat flow with respect to the $L^2$-Kantorovich distance. Moreover one immediately recovers the $L^2$-transport and gradient estimates
\begin{align*}
W_2(P_t\mu,P_t\nu)\leq e^{-Kt}W_2(\mu,\nu),\quad \Gamma(P_tu)\leq e^{-2Kt}P _t\Gamma(u),
\end{align*}
where $\Gamma$ denotes the \emph{Carr\'e du champ} operator. On a Riemannian manifold $(M,g)$, $\Gamma(u)$ corresponds to $|\nabla u|_g^2$.

Savar\'e proved in \cite{savare} that the gradient estimate implies the stronger estimate
\begin{align*}
( \Gamma(P_tu))^\alpha\leq e^{-2\alpha Kt}P _t(\Gamma(u)^\alpha),
\end{align*}
where $\alpha\in[1/2,2]$. This has been first noticed by Bakry in \cite{bakrytrans} in the framework of Dirichlet forms. Crucial for this estimate is the self-improvement of the \emph{Bochner inequality}. By Kuwada's duality \cite{kuwadadual} the stronger gradient estimate further implies stronger $L^p$-transport estimates
\begin{align*}
W_p(P_t\mu,P_t\nu)\leq e^{-Kt}W_p(\mu,\nu) \quad p\in[1,\infty].
\end{align*}

In this paper we aim to show similar estimates for time-dependent metric measures spaces $(X,d_t,m_t)_{t\in I}$, $I=(0,T)$, which evolve under a super-Ricci flow. Moreover we introduce Brownian motions, proving their existence and uniqueness.
We parametrize them backward in time, which seems to be unconventional at first glance, but is the right thing to do when
considering super-Ricci flows, which we parametrize forward in time. We refer to this parametrization of Brownian motion
as \emph{backward Brownian
motion} and prove transport estimates of their trajectories. 

$(X,d_t,m_t)_{t\in I}$ is said to be a super-Ricci flow if
for a.e.\ $t\in I$ and every  $W_t$-geodesic  $(\mu^a)_{a\in[0,1]}$ in $\mathcal P(X)$
\begin{equation*}
\partial_a S_t(\mu^{a})\big|_{a=1}-\partial_a S_t(\mu^{a})\big|_{a=0}
\ge- \frac 12\partial_t^-W_{t}^2(\mu^0,\mu^1),
\end{equation*}
where $S_t$ denotes the Boltzmann entropy at time $t$ and $W_t=W_{2,t}$ the $L^2$-Kantorovich distance with respect to $d_t$.
 This notion goes back to Sturm in \cite{sturm2015} and is a generalization to super-Ricci flows for manifolds (``$\Ric_t\geq-\frac12\partial_t g_t$") on one hand and to CD$(K,\infty)$-spaces on the other. 

In \cite{sturm2016} the authors show existence and uniqueness of the heat flow $(P_{t,s})_{t\geq s}$ on $(X,d_t,m_t)_{t\in I}$, i.e. $P_{t,s}u$ solves
\begin{align*}
\partial_tu_t=\Delta_tu_t,\text{ for }t>s \text{ and } u_s=u,
\end{align*}
 and moreover, that super-Ricci flows are characterized by the time-dependent gradient estimate
\begin{align*}
 \Gamma_t(P_{t,s}u)\leq P_{t,s}(\Gamma_s(u)),
\end{align*}
or equivalently, by the $L^2$-Kantorovich transport estimate
\begin{align}\label{transporting}
W_{2,s}(\hat P_{t,s}\mu,\hat P_{t,s}\nu)\leq W_{2,t}(\mu,\nu),                       
\end{align}
where $\hat P_{t,s}$ denotes the dual heat flow on probability measures. See also Definition \ref{supereva} below. The spaces $(X,d_t,m_t)_{t\in I}$ are assumed to satisfy RCD$(K,N)$ for some $K,N\in\mathbb N$, i.e. they satisfy RCD$(K,\infty)$ with dimensions bounded from above by $N$. For more details we refer to Section \ref{sec:proofofmain}.

In the same setting as in \cite{sturm2016} we will introduce backward Brownian motions (see Definition \ref{defofbm}) and prove their existence and uniqueness (Proposition \ref{prop:exbm}). Moreover we construct couplings of backward Brownian motions satisfying a pathwise contraction estimate assuming that stronger $L^p$-transport estimates hold
 (Theorem \ref{brownian}).
\begin{thm}\label{thmb}
Let $(X,d_t,m_t)_{t\in I}$ be a family of RCD$(K,N)$-spaces.

Then there exists a unique backward Brownian motion $(B_s)_{s\leq t}$ with terminal distribution $\mu$, i.e. a sample-continuous Markov process
with transition probabilities $\hat P_{s',s}(\delta_x)$.

Moreover, assuming that
\begin{align}\label{eq:strongtransport}
W_{p,s}(\hat P_{t,s}\mu,\hat P_{t,s}\nu)\leq W_{p,t}(\mu,\nu),
\end{align}
holds for all $s\leq t$, $p\in[0,\infty]$ and $\mu,\nu\in\mathcal P(X)$, there exists a coupling $(B_s^1,B_s^2)_{s\leq t}$ of backward Brownian motions satisfying 
\begin{align}\label{eq:strongtransport2}
d_s(B_s^1,B_s^2)\leq d_t(B_t^1,B_t^2)\text{ almost surely}.
\end{align}
\end{thm}

The question is under what circumstances the estimate in \eqref{eq:strongtransport} is satisfied in the setting of 
time-dependent metric measure spaces. The answer is it holds if and only if $(X,d_t,m_t)$ is a super-Ricci flow.
Moreover the estimate in \eqref{eq:strongtransport2} characterize super-Ricci flows as well.
This is summarized in the following theorem. For the precise statement see Theorem \ref{thm3}.

 \begin{thm}\label{thm3intro}
Let $(X,d_t,m_t)_{t\in I}$ be a family of RCD$(K,N)$-spaces such that $t\mapsto\Gamma_t(u)$ is a $\mathcal C^1$-function for all $u\in\Lip(X)$. Then, $(X,d_t,m_t)_{t\in I}$ is a super-Ricci flow if and only if one of the following equivalent properties holds
\begin{enumerate}
\item[i)] for all $s\leq t$, $\alpha\in[1/2,1]$ and $u\in \Lip(X)$
 \begin{align*}
 (\Gamma_t(P_{t,s}u))^\alpha\leq P_{t,s}(\Gamma_s(u)^\alpha)\quad m\text{-a.e.},
     \end{align*}  
\item[ii)] for all $s\leq t$, $p\in[1,\infty]$ and $\mu,\nu\in\mathcal P(X)$
\begin{align*}
W_{p,s}(\hat P_{t,s}\mu,\hat P_{t,s}\nu)\leq W_{p,t}(\mu,\nu),
\end{align*}
\item[iii)] there exists a coupling of backward Brownian motions
$(B_s^1,B_s^2)_{s\leq t}$ such that for all $s\leq t$
\begin{align*}
 d_{s}(B_{s}^1,B_{s}^2)\leq d_t(B_t^1,B_t^2) \text{ almost surely}.
\end{align*}
\end{enumerate}
\end{thm}
The proof of this is subdivided into developing two results, Theorem \ref{thm2intro} and Theorem \ref{thm1intro}, which we will
state separately, since they are interesting in itself. In Theorem \ref{thm2intro} we derive a pointwise dynamic version of Bochner's inequality,
which is crucial to deduce the improved gradient estimate in Theorem \ref{thm1intro}. Having established the improved gradient
estimate, Kuwada's duality \cite[Theorem 2.2]{kuwadadual} and Theorem \ref{thmb} immediately imply Theorem \ref{thm3intro}. 
\begin{thm}\label{thm2intro}
Let $(X,d_t,m_t)_{t\in I}$ be a family of RCD$(K,N)$-spaces such that $t\mapsto\Gamma_t(u)$ is a $\mathcal C^1$-function 
for all $u\in\Lip(X)$. Then, if $(X,d_t,m_t)_{t\in I}$ is a super-Ricci flow the pointwise dynamic Bochner inequality
\begin{align*}
 \frac12\Delta_t(\Gamma_t(u))-\Gamma_t(u,\Delta_tu)\geq \frac12(\partial_t\Gamma_t)(u).
\end{align*}
 holds for all $t\in I$.
\end{thm}

For the precise statement see Theorem \ref{thm2}.
Note that the authors in \cite{sturm2016} derive a dynamic version of Bochner's inequality, where 
the test functions $u$ appear as heat flows $P_{t,s}u_s$ and the estimate holds in an almost everywhere sense.
The pointwise dynamic Bochner inequality is an essential modification, since from this we obtain the improved gradient 
estimates $i)$ in Theorem \ref{thm3intro}. We summarize this in the following theorem, see also Theorem \ref{thm1}.
\begin{thm}\label{thm1intro}
Let $(X,d_t,m_t)_{t\in I}$ be a family of RCD$(K,N)$-spaces such that $t\mapsto\Gamma_t(u)$ is a $\mathcal C^1$-function for all $u\in\Lip(X)$. Then, if $(X,d_t,m_t)_{t\in I}$ is a super-Ricci flow 
for every $\alpha\in[1/2,1]$ we have for $s\leq t$ and every $u\in\Lip(X)$
 \begin{align*}
 (\Gamma_t(P_{t,s}u))^\alpha\leq P_{t,s}(\Gamma_s(u)^\alpha) \quad m\text{-a.e..}
     \end{align*} 
\end{thm}

Let us remark that several difficulties arise due to the time-dependence, especially regarding the regularity of the heat flow, e.g. the Laplacian and the semigroup do not commute and the domain of the Laplacian is time-dependent.


In the following we visualize super-Ricci flows satisfying the properties in Theorem \ref{thm3intro} by giving a 
few examples:
\begin{ex}
Static spaces: Let $(X,d,m)$ be an RCD$(K,N)$-space. Then the constant space $(X,d_t,m_t)_{t\in I}$ is a super-Ricci flow 
if and only if $K=0$.
\end{ex}
\begin{ex}
Smooth spaces: Let $(M,g_t,\textrm{vol}_t)_{t\in I}$ be a family of smooth manifolds, where $(g_t)_{t\in I}$ is a family 
of
smooth Riemannian metrics with associated volume measures $(\textrm{vol}_t)_{t\in I}$. The metrics satisfy the super-Ricci flow 
estimate
\begin{align*}
\Ric_t\geq-\frac12\partial_tg_t
\end{align*}
if and only if $i),ii)$ or $iii)$ of Theorem \ref{thm3intro} is satisfied. See also Corollary 1 in \cite{mccanntopping}, 
Theorem 4.1 in \cite{acthorizontal}, and Theorem 1.5 in \cite{HN2015}.

E.g. the Ricci flow $\Ric_t=-\frac12\partial_tg_t$ introduced by Hamilton in \cite{Ham1} defines a super-Ricci flow. But
this ``minimal'' Ricci flow is not induced by $i),ii)$ or $iii)$ of Theorem \ref{thm3intro}. Haslhofer and Naber characterize
Ricci flows via functional inequalities on the path space,
see Theorem 1.22 in \cite{HN2015}.
\end{ex}
\begin{ex}
Smooth weighted spaces: Let $(M,g_t,m_t)_{t\in I}$ be a family of smooth manifolds, where 
$(g_t)_{t\in I}$ are smooth Riemannian metrics and the volume measures $(m_t)_{t\in I}$ are given by
$dm_t=e^{-f_t}\, d\textrm{vol}_t$. Then, $g_t$ and $f_t$ satisfy the estimate
\begin{align*}
\Ric_t+\nabla^2_t f_t\geq-\frac12\partial_tg_t,
\end{align*}
if and only if $i),ii)$ or $iii)$ of Theorem \ref{thm3intro} is satisfied,
cf. Theorem 0.2 in \cite{sturm2015}.
\end{ex}
\begin{ex}
Wandering Gaussian: Let $(X,d_t,m_t)_{t\in I}$ such that $X=\mathbb R^n$, $d_t(x,y)=||x-y||$ and $dm_t=e^{-f_t}dx$ with
\begin{align}
f_t(x)=\langle x,\alpha_t\rangle^2+\langle x,\beta_t\rangle +\gamma_t,
\end{align}
where $\alpha,\beta\colon I\to\mathbb R^n$ and $\gamma\colon I\to \mathbb R$ are arbitrary functions. Then this is a super-Ricci flow 
if and only if $i),ii)$ or $iii)$ of Theorem \ref{thm3intro} is satisfied.
\end{ex}
\begin{ex}
Homothetic flows: Let $(X,d_t,m_t)_{t\in I}$, where $m_t=m$ and 
\begin{align*}
d_t^2=d^2(1-2Kt),
\end{align*}
 such that $(X,d,m)$ satisfies RCD$(K,N)$. Then this is a super-Ricci flow if and only if
 $i),ii)$ or $iii)$ of Theorem \ref{thm3intro} is satisfied, cf. Proposition 2.7 in \cite{sturm2015} and Example \ref{homo} below.
\end{ex}
  \begin{ex}\label{homo}
  A non-smooth example for an homothetic flow can be constructed on
  the spherical cone $\Sigma$ over $M=S^2(1/\sqrt{3})\times S^2(1/\sqrt{3})$ with product metric $d$ as in \cite{sturm2016}: The spherical cone is defined as the quotient of $M\times [0,\pi]$ by contracting $M\times\{0\}$ to the south pole $S$ and $M\times \{\pi\}$ to the north pole $N$ with metric
\begin{align}
\cos(d_\Sigma((x,s),(x',s'))):=\cos s\cos s'+\sin s\sin s'\cos(d(x,x')\wedge\pi)
\end{align}
and measure $dm_\Sigma(x,s):=d\textrm{vol}(x)\otimes(\sin^4 s\, ds)$.

  This space is a RCD$(4,5)$-space (see \cite[Thorem 3.5]{bacher}), and the punctured cone $\Sigma_0:=\Sigma\setminus\{S,N\}$
  is a 5-dimensional (non-complete) Riemannian manifold with metric $g_0$ and constant curvature $\Ric_{g_0}=4$ 
  (see \cite[Lemma 3.6]{bacher}). A possible Ricci flow on the punctured cone is then given by
\begin{align*}
g_t=(1-8t)g_0.
\end{align*}
The associated metric measure spaces of $(\Sigma_0,g_t,\textrm{vol}_t)_{t\in I}$ build a super-Ricci flow for all 
$\bar I\subset [0,1/8)$, since the geodesics over the poles play no role in optimal transport, cf. \cite[Theorem 3.2]{bacher}. Let us remark that the sectional curvature of the punctured spherical cone is neither bounded from below nor from above.
  \end{ex}

Let us briefly make some remarks about related literature.

Ricci flows have been introduced by Hamilton in \cite{Ham1}. They gained a lot of attention when
Perelman utilized them for proving the Poincar\'e conjecture for 3-manifolds with positive Ricci curvature \cite{perelman2002entropy,perelman2003ricci,perelman2003finite}. 

In respect of optimal transport, McCann and Topping \cite{mccanntopping} showed that the $L^2$-transport estimate 
holds if and only if the underlying compact manifolds evolve as a super-Ricci flow, indicating how super-Ricci flows should be defined in a weaker context.

Coming from a probabilistic viewpoint, Arnaudon, Coulibaly and Thalmaier \cite{acthorizontal} refined McCann's and Topping's result by showing a pathwise estimate: They prove the existence of coupled Brownian motions with pathwise contraction. Their result implies the stronger $L^1$-gradient estimate.

Kuwada and Philipowski \cite{kuwada2011coupling} construct couplings of Brownian motions such that the normalized Perelman's $\mathcal L$-distance of the coupling is a supermartingale, which implies transport estimates for more general costs, see also the results in \cite{toppingoptimal}.
This construction is obtained on smooth Riemannian manifolds evolving as a backward super-Ricci flow.

Haslhofer and Naber characterize in \cite{HN2015} Ricci flows in terms of functional inequalities on the path space of smooth
Riemannian manifolds. Inserting one-point cylinder functions implies the characterization of super-Ricci flows in terms of a gradient estimate with $\alpha=1$ and $\alpha=1/2$
and in terms of a Bochner formula. \\

The plan of the paper goes as follows. First we will recall the notion of heat flows on time-dependent metric measure
spaces introduced in \cite{sturm2016} and state the characterization of super-Ricci flows in terms of the heat flow. 
In Section \ref{sec:brownian} we introduce backward Brownian motions on time-dependent metric measure spaces and prove 
their existence and uniqueness. We anticipate the $p$-transport estimates \eqref{eq:strongtransport} and prove that 
they imply contraction estimates for backward Brownian motions. The proof of Theorem \ref{thm3intro} is arranged in 
Section \ref{sec:transport}. For this we first introduce a new dynamic Bochner inequality from which we deduce the 
gradient estimate with exponent $\alpha\in[1/2,1]$. In Section \ref{sec:bochner} we show that this dynamic Bochner 
inequality indeed holds under super-Ricci flow. Finally we gather in Section \ref{sec:proofofeq} the last steps to 
prove Theorem \ref{thm3intro}.

\section{Preliminaries}\label{sec:proofofmain}
In the sequel let $(X,d_t,m_t)_{t\in I}$, where $I=(0,T)$, be a one-parameter family of geodesic Polish metric measure spaces such that the following holds:
\begin{enumerate}
 \item There exists a finite reference measure $m$ with full topological support such that $m_t=e^{-f_t}m$ with Borel functions $(f_t)$ satisfying 
 \begin{align}\label{assumption1}
|f_t(x)|\leq C,\quad |f_t(x)-f_t(y)|\leq Cd_t(x,y),\quad|f_t(x)-f_s(x)|\leq L|t-s|,
 \end{align}
with constants $C,L>0$ independent of $x,y\in X$ and $s,t\in I$.
\item the distance is ``log-Lipschitz'' continuous, i.e.
\begin{align}\label{assumption2}
 |\log(d_t(x,y)/d_s(x,y))|\leq L|t-s|
\end{align}
for all $x,y\in X$ and all $s,t\in I$,
\item there exist constants $K,N\in\mathbb R$ such that for each $t\in I$
 the space $(X,d_t,m_t)$ satisfies the Riemannian curvature-dimension bound RCD$(K,N)$.
\end{enumerate}

RCD$(K,N)$-spaces are metric measure spaces which are infinitesimally Hilbertian and satisfy the curvature dimension condition 
CD$(K,N)$. CD$(K,N)$-spaces have been introduced independently by Sturm \cite{sturm2006} and Lott-Villani \cite{lott2009ricci}.
In order to rule out Finsler spaces Ambrosio, Gigli and Savar\'e introduced in \cite{agsmet} the notion
infinitesimally Hilbertian, also called Riemannian. CD$(K,N)$-spaces are stable under measured Gromov-Hausdorff convergence,
infinitesimally Hilbertian spaces not. As observed in \cite{giglionthe} RCD$(K,N)$-spaces are stable, which has been finally 
proven in 
\cite{eks2014}.

Under mild regularity assumptions ``$(X,d_t,m_t)$ satisfies RCD$(K,N)$" means that $(X,d_t,m_t)$ is infinitesimally 
Hilbertian, i.e. $\Delta_t$ is a linear operator, and
\begin{align*}
\frac12\int \Delta_t g\Gamma_t(f)\, dm-\int g\Gamma_t(\Delta_t f,f)\, dm\geq K\int g\Gamma_t(f)\, dm+\frac1N\int g(\Delta_t f)^2\, dm
\end{align*}
holds for all $g\in \Dom(\Delta_t)$ bounded and nonnegative with $\Delta_t g$ bounded and all $f\in \Dom(\Delta_t)$ 
with $\Delta_t f\in W^{1,2}(X,m)$, see \cite {eks2014}. 

In the following we will explain the metric setup in more detail, i.e. briefly recalling the main ideas 
from \cite{agscalc,agsmet, agsbe, giglionthe}.

The Cheeger energy $\Ch_t$ at time $t\in I$ is defined as the convex and lower-semicontinuous functional in $L^2(X,m_t)$ 
\begin{align*}
\Ch_t(u):=\inf\Big\{\liminf_{n\to\infty}\frac12\int_X\lip_t(u_n)^2\, dm_t\Big\}
\end{align*}
where the infimum is taken over all bounded Lipschitz functions $u_n\in\Lip_b(X)$ such that $u_n\to u$ in $L^2(X,m_t)$ (cf. \cite{agscalc,sturm2015}). Here, $\lip_tu$ denotes the local Lipschitz constant w.r.t. the metric $d_t$
\begin{align*}
\lip_tu(x):=\limsup_{y\to x}\frac{|u(y)-u(x)|}{d_t(x,y)},
\end{align*}
and
$\Ch_t$ admits the local representation formula
\begin{align*}
\Ch_t(u)=\frac12\int_X|\nabla_tu|_*^2\, dm_t,
\end{align*}
where $|\nabla_tu|_*$ is the minimal relaxed gradient \cite{agscalc}.
Since $(X,d_t,m_t)$ satisfies a Riemannian curvature bound, (in particular $\Ch_t$ is quadratic) $\E_t:=2\Ch_t$ is a strongly local Dirichlet form with Carr\'e du Champ 
\begin{align*}
\Gamma_t(u)=|\nabla_tu|^2_*
\end{align*}
cf. \cite{savare,agsbe,agsmet}, i.e.
\begin{align}\label{gamma}
 \E_t(u)=\int_X\Gamma_t(u)\, dm_t.
\end{align}
Thanks to \eqref{gamma}, $\E(u,v)=\int_X\Gamma_t(u,v)\, dm_t$ where
\begin{align*}
 \Gamma_t(u,v):=\frac14(\Gamma_t(u+v)-\Gamma_t(u-v)).
\end{align*}
$\Gamma(\cdot,\cdot)$ satisfies the chain rule and the Leibniz rule
\begin{align*}
 \Gamma_t(\theta(u),v)=\theta'(u)\Gamma_t(u,v),\quad \Gamma_t(uv,w)=u\Gamma_t(v,w)+v\Gamma_t(u,w),
\end{align*}
where $u,v,w\in\Dom(\E_t)$ and $\theta\in\Lip(\mathbb R)$, $\theta(0)=0$.
We call the linear generator $\Delta_t$ the Laplacian and 
$$-\int_X \Delta_tu \, v\,dm_t=\E_t(u,v)\qquad\forall u\in\Dom(\Delta_t), v\in \Dom(\E_t),$$
with domain $\Dom(\Delta_t)\subset\Dom(\E_t)$.

Due to our assumptions \eqref{assumption1} and \eqref{assumption2}, the sets $L^2(X,m_t)$ and $W^{1,2}(X,d_t,m_t):=\mathcal D(\E_t)$ do not depend on $t$ 
and the respective norms for varying $t$ are equivalent to each other.
We put  $\mathcal H= L^2(X,m)$ and  $\F=\Dom(\E_{t_0})$ for some fixed $t_0$
as well as  $$\F_{(s,\tau)}=L^2\big((s,\tau)\to \F\big)\cap  H^1\big( (s,\tau)\to \F^*\big)\subset {\mathcal C}\big([s,\tau]\to \mathcal H\big)$$ for each 
$0\le s<\tau\le T$.

\subsubsection*{The heat equations}
A function $u$ is called \emph{solution to the heat equation}
$$\Delta_t u=\partial_t u\qquad\mbox{on }(s,\tau)\times X$$
if $u\in \F_{(s,\tau)}$ and if for all $w\in \F_{(s,\tau)}$
\begin{equation}\label{def-heat}
-\int_s^\tau \E_t(u_t,w_t)dt
=\int_s^\tau \langle\partial_tu_t, w_te^{-f_t}\rangle_{\mathcal F^*,\mathcal F}\,dt
\end{equation}
where $\langle\cdot,\cdot\rangle_{\mathcal F^*,\mathcal F}=\langle\cdot,\cdot\rangle$ denotes the dual pairing. Note that thanks to \eqref{assumption1}, 
$w\in L^2\big((s,\tau)\to \F\big)$ if and only if $we^{-f}\in L^2\big((s,\tau)\to \F\big)$.

Further a function $v$ is called \emph{solution to the adjoint heat equation}
$$-\Delta_s v +\partial_s f \cdot v=\partial_s v\qquad\mbox{on }(\sigma,t)\times X$$
if $v\in \F_{(\sigma,t)}$ and if for all $w\in \F_{(\sigma,t)}$
$$\int_\sigma^t \E_s(v_s,w_s)ds+\int_\sigma^t \int_X v_s\cdot w_s\cdot \partial_sf_s\,dm_s\,ds=\int_\sigma^t \langle\partial_sv_s, w_se^{-f_s}\rangle_{\F^*,\F}\,ds.$$

We recall the following results from \cite{sturm2016}.

\begin{thm}\label{thm:kernels}
(i) For each $0\le s<\tau\le T$ and each $h\in \mathcal H$ there exists a unique solution  $u\in \F_{(s,\tau)}$ to the heat equation $\partial_tu_t=\Delta_tu_t$ on $(s,\tau)\times X$ with $u_s=h$.

(ii) The heat propagator $P_{t,s}: h\mapsto u_t$ admits a kernel $p_{t,s}(x,y)$ w.r.t.\ $m_s$, i.e.
\begin{equation}
\label{heat-kernel}P_{t,s}h(x)=\int p_{t,s}(x,y) h(y)\, dm_s(y).
\end{equation}
If $X$ is bounded, for each $(s',y)\in (s,T)\times X$ the function $(t,x)\mapsto p_{t,s}(x,y)$ is a solution to the heat equation on $(s',T)\times X$.

(iii) All  solutions $u: (t,x)\mapsto u_t(x)$ to the heat equation on $(s,\tau)\times X$ are H\"older continuous in $t$ and $x$. All nonnegative solutions satisfy a scale invariant parabolic Harnack inequality of Moser type.

(iv) The heat kernel $p_{t,s}(x,y)$ is H\"older continuous in all variables, it is Markovian 
$$\int p_{t,s}(x,dy):=\int p_{t,s}(x,y)\,dm_s(y)=1\qquad\quad (\forall s<t, \forall x)$$
and has the propagator property
$$p_{t,r}(x,z)=\int p_{t,s}(x,y)\, p_{s,r}(y,z)\, dm_s(y)\qquad\quad (\forall r<s<t, \forall s,z).$$
\end{thm}
\begin{thm} 
(i) For each $0\le\sigma<t\le T$ and each $g\in \mathcal H$ there exists a unique solution  $v\in \F_{(0,t)}$ to the adjoint heat equation $\partial_sv_s=-\Delta_sv_s+(\partial_s f_s)v_s$ on $(\sigma,t)\times X$ with $v_t=g$.

(ii) This solution is given as $v_s(y)= P^*_{t,s}g(y)$ in term of the adjoint  heat propagator 
\begin{equation}
 P^*_{t,s}g(y)=\int p_{t,s}(x,y) g(x)\, dm_t(x).
\end{equation}
If $X$ is bounded, for each $(t',x)\in (0,t)\times X$ the function $(s,y)\mapsto p_{t,s}(x,y)$ is a solution to the adjoint heat equation on $(0,t')\times X$.

(iii) All  solutions $v: (s,y)\mapsto v_s(y)$ to the adjoint heat equation on $(\sigma,t)\times X$ are H\"older continuous in $s$ and $y$. All nonnegative solutions satisfy a scale invariant parabolic Harnack inequality of Moser type.
\end{thm}

By duality, the propagator $(P_{t,s})_{s\le t}$ acting on bounded continuous functions induces a \emph{dual propagator} $(\hat P_{t,s})_{s\le t}$ acting on 
probability measures as follows
\begin{equation*}
\int u\, d(\hat P_{t,s}\mu)=\int (P_{t,s}u)d\mu\qquad\forall u\in{\mathcal C}_b(X),  \forall \mu\in\mathcal P(X),
\end{equation*}
where $\mathcal P(X)$ denotes the space of all Borel probability measures.

The time-dependent function $v_t(x)= P_{t,s}u(x)$ is a solution to the  heat equation, whereas
the time-dependent measure $\nu_s(dy)=\hat P_{t,s}\mu(dy)$ is a solution to the \emph{dual heat equation}
$$-\partial_s \nu=\hat\Delta_s \nu.$$
Again $\hat \Delta_s$ is defined by duality:
$\int u\, d(\hat \Delta_s\mu)=\int \Delta_s u \,d\mu\quad\forall u, \forall\mu.$

\begin{lma}[Lemma 3.8 in \cite{sturm2016}]\label{P*Eva}
 Let $u,g\in\F$ and $t\in I$ with $g\in L^1(X,m_t)$. 
 Then, 
 \begin{align*}
  \lim_{h\searrow0}\frac1h\left(\int ugdm_t-\int uP_{t,t-h}^*gdm_{t-h}\right)=\int\Gamma_t(u,g)dm_t.
 \end{align*}

\end{lma}
\begin{thm}[Theorem 2.12 in \cite{sturm2016}] \label{energy-estEva}
 For all $0<s<\tau<T$ and for all solutions $u\in \F_{(s,T)}$ to the heat equation
\begin{itemize}
\item[(i)] 
$u_t\in\Dom(\Delta_t)$ for a.e.\ $t\in (s,\tau)$.

\item[(ii)]
If the initial condition $u_s\in\F$ then
$$u\in L^2\big((s,\tau)\to \Dom(A_\cdot\big)\cap H^1\big( (s,\tau)\to \mathcal H\big).$$
More precisely,
\begin{equation}\label{tre}
e^{-3L\tau}\E_\tau(u_\tau)+2\int_s^\tau e^{-3Lt}\int_X \big| \Delta_tu_t\big|^2\,dm_t\,dt\le  e^{-3Ls}\cdot \E_s(u_s).
\end{equation}

\item[(iii)]For all solutions $v$ to the adjoint heat equation on $(\sigma,t)\times X$ and all $s\in (\sigma,t)$
$$\E_s(v_s)+\|v_s\|^2_{L^2(m_s)}\le e^{3L(t-s)}\cdot\Big[  \E_t(v_t)+\|v_t\|^2_{L^2(m_t)}\Big].$$
Moreover, 
$v_s\in\Dom(\Delta_s)$
for a.e.\ $s\in(\sigma,t)$.
\end{itemize}

\end{thm}

\subsubsection*{Super-Ricci flows}
Let us recall the notion of super-Ricci flows as defined in \cite{sturm2015}. We denote the relative entropy by
$S\colon I\times\mathcal P(X)\to(-\infty,\infty]$, i.e.
\begin{align*}
 S_t(\mu)=\int \rho\log\rho\, dm_t
\end{align*}
whenever $\mu=\rho m_t$, and $S_t(\mu)=\infty$ otherwise.

We set for each $p\in[1,\infty)$
\begin{align*}
 W_{p,t}(\mu_1,\mu_2)=\min\left\{\int_{X\times X}d_t^p(x,y)\, d\gamma(x,y)|\gamma\in \Pi(\mu_1,\mu_2)\right\}^{1/p},
\end{align*}
where $\Pi(\mu_1,\mu_2)$ is the space of all measures in $\mathcal P(X\times X)$ whose marginals $(e_i)_\#\mu$ coincide with $\mu_i$. We also set
\begin{align*}
W_{\infty,t}(\mu_1,\mu_2)=\inf\left\{||d_t||_{L^\infty(\gamma)}|\gamma\in\Pi(\mu_1,\mu_2)\right\}=\lim_{p\to\infty}W_{p,t}(\mu_1,\mu_2),                                                                                                                                                   
\end{align*}
with essential supremum $||d||_{L^\infty(\gamma)}=\inf\{C\geq0|d(x,y)\leq C\, \gamma\text{-a.e. }x,y\}$. For the second equality see e.g. Lemma 3.2 in \cite{kuwadadual}.

Contraction estimates for the dual heat flow under $W_{2,t}=W_t$ characterize, among others, super-Ricci flows in the sense of \cite{sturm2015}.  This is the main theorem in \cite{sturm2016} and is repeated in the following definition. 

\begin{Def}[Theorem 1.7 in \cite{sturm2016}]\label{supereva}
 We say that $(X,d_t,m_t)_{t\in(0,T)}$ is a \emph{super-Ricci flow} if one of the following equivalent assertions holds
 \begin{enumerate}
  \item[i)] For a.e.\ $t\in (0,T)$ and every  $W_t$-geodesic  $(\mu^a)_{a\in[0,1]}$ in $\mathcal P(X)$ with
$\mu^0,\mu^1\in \Dom(S)$
\begin{equation}\label{est-Ieva}
\partial^+_a S_t(\mu^{a})\big|_{a=1-}-\partial^-_a S_t(\mu^{a})\big|_{a=0+}
\ge- \frac 12\partial_t^- W_{t-}^2(\mu^0,\mu^1)
\end{equation}
(`dynamic convexity').
\item[ii)]
For all $0\le s<t\le T$ and $\mu,\nu\in\mathcal P(X)$
\begin{equation}\label{est-IIeva}
W_s (\hat P_{t,s}\mu,\hat P_{t,s}\nu)\le W_t (\mu,\nu)
\end{equation}
(`transport estimate').
\item[iii)]
For all $u\in\Dom(\E)$ and all $0<s<t< T$
\begin{equation}\label{est-IIIeva}
\big|\nabla_t(P_{t,s}u)\big|_*^2\le P_{t,s}\big(|\nabla_s u|_*^2\big)
\end{equation}
(`gradient estimate').
\item[iv)] 
For all $0<s<t<T$ and for all $u_s,g_t\in\F$ with $g_t\ge0$,  $g_t\in L^\infty$, $u_s\in \Lip(X)$
and for a.e.\ $r\in(s,t)$ 
 \begin{equation}\label{est-IVeva}
 {\bf \Gamma}_{2,r}(u_r)(g_r)
 \geq\frac12\int\stackrel{\bullet}{\Gamma}_r(u_r)g_rdm_r
\end{equation}
(`dynamic Bochner inequality' or `dynamic Bakry-Emery condition')
where $u_r=P_{r,s}u_s$ and $g_r=P^*_{t,r}g_t$.

 \end{enumerate}
Here 
 $${\bf \Gamma}_{2,r}(u_r)(g_r):= \int\Big[\frac12\Gamma_{r}(u_r)\Delta_r g_r
 +(\Delta_r u_r)^2g_r
 +\Gamma_r(u_r,g_r)\Delta_ru_r\Big]dm_r$$ 
 denotes the distribution valued $\Gamma_2$-operator (at time $r$) applied to $u_r$ and tested against $g_r$
 and
$$\stackrel{\bullet}{\Gamma}_r(u_r):=\mbox{w-}\lim_{\delta\to0}\
\frac1\delta\Big(\Gamma_{r+\delta}(u_r)-\Gamma_r(u_r)\Big)$$ 
denotes any subsequential weak limit  of $\frac1{2\delta}\big(\Gamma_{r+\delta}-\Gamma_{r-\delta}\big)(u_r)$ in $L^2((s,t)\times X)$.

\end{Def}

\section{Brownian motions and transport estimates}\label{sec:brownian}

In the following we introduce backward Brownian motions on time-dependent metric measure spaces $(X,d_t,m_t)_{t\in I}$ from Section \ref{sec:proofofmain}, and prove their existence.

Recall that by Theorem \ref{thm:kernels} $(p_{t,s})_{s<t}$ are Markov kernels which satisfy the propagator property
\begin{align*}
p_{t,r}(x,z)=\int p_{t,s}(x,y)\, p_{s,r}(y,z)\, dm_s(y)\qquad\quad (\forall r<s<t, \forall s,z).
\end{align*}

\begin{Def}\label{defofbm}
Let $\mu\in\mathcal P(X)$ and $t\in I$.
 We call a stochastic process $(B_s)_{s\leq t}$ on a probability space $(\Omega,\Sigma,\mathbb P)$ with values in $X$
a \emph{backward Brownian motion} on $X$ with terminal distribution $\mu$ if it is a sample-continuous Markov process with transition probabilities 
 \begin{align*}
  \mathbb P[B_s\in A|B_{s'}=x]:=\hat P_{s',s}(\delta_x)(A)=\int_Ap_{s',s}(x,y)\, dm_s(y)
 \end{align*}
 for all $s<s'< t$ and every Borel set $A$ such that $\mathbb P\circ B_t^{-1}=\mu$.
\end{Def}
\begin{remark*}
We define backward Brownian motions with terminal data, since we later want to use them to characterize super-Ricci flows (Theorem \ref{thm3}).
 Parametrizing them forward in time, we would have to consider backward super-Ricci flows in order to get the equivalent result.
\end{remark*}
\begin{remark*}
 We consider the time-dependent generator $\Delta_s$ instead of $\frac12\Delta_s$.
 This is only for convenience and the stochastic process with generators $(\frac12\Delta_s)_{s\leq 2t}$ is given by 
 $(\tilde B_{s})_{s\leq 2t}$,
 where $\tilde B_s:=B_{\frac s2}$.
\end{remark*}

In order to prove existence of backward Brownian motions we consider for fixed $t\in I$ the finite subset 
$J=\{t_1,\cdots,t_r\}$ of $(0,t]$, where $0<t_1<t_2<\ldots<t_r<t$ and the
finite dimensional distribution $P_J^\mu$, where $\mu\in \mathcal P(X)$, defined by
\begin{align*}
 &P_J^\mu(A_r\times\ldots\times A_1)\\
 :=&\int_X\int_{A_r}\ldots\int_{A_1} p_{t_2,t_1}(x_{t_2},x_{t_1})\, dm_{t_1}(x_{t_1})\ldots p_{t,t_{r}}(x,x_{t_{r}})\, dm_{t_{r}}(x_{t_r})\, d\mu(x).
\end{align*}
The family of probability measures $\{P_J^\mu|J\text{ finite }\subset (0,t]\}$ defines a projective family, hence the Kolmogorov extension theorem 
\cite[Theorem 35.5]{bauer}
implies that there exists a unique probability measure 
$P_{(0,t]}^\mu$ on $(X^{(0,t]},\mathcal B(X)^{(0,t]})$
such that $(\pi_J)_\#P_{(0,t]}^\mu=P_J^\mu$. Here, $\pi_J$ denotes the projection $\omega\mapsto (\omega(t_1),\ldots,\omega(t_r))$ from $X^{(0,t]}$ to $X^r$.

  For every $s\in (0,t]$ the map $\pi_s\colon \omega\mapsto \omega(s)$ from $X^{(0,t]}$ to $X$ is a stochastic process with
finite-dimensional distributions $(P_J^\mu)_J$. The following Proposition yields existence of a continuous modification
$(B_s)_{s\leq t}$, and hence a backward Brownian motion.
\begin{Prop}\label{prop:exbm}
For each $t\in I$ and each $\mu\in\mathcal P(X)$ there exists a backward Brownian motion on $X$ with terminal distribution $\mu$, which is unique in law.
\end{Prop}
\begin{proof}
 We need to show that there exists positive constants $\alpha, \beta,c>0$ such that the above mentioned process $\pi_s$ satisfies
 \begin{align}\label{eq:continuitylp}
 E[d(\pi_{s'},\pi_s)^\alpha]\leq c|s-s'|^{1+\beta}
 \end{align}
for all $s',s\in(0,t]$. Then the Kolmogorov continuity theorem \cite[Theorem 39.3]{bauer} implies that there exists a modification $(B_s)_{s\leq t}$ such that
the map $s\mapsto B_s(\omega)$ is continuous for $P_{(0,t]}^\mu$-a.e $\omega$. Hence the process $(B_s)_{s\leq t}$ on the probability space
$(X^{(0,t]},\mathcal B(X)^{(0,t]},P_{(0,t]}^\mu)$ yields the desired properties. For $\alpha>2$ \eqref{eq:continuitylp}
follows from \eqref{continuitylpproof} in the proof of Lemma \ref{continuitylp} below.

Since all finite-dimensional distributions are uniquely determined this process is unique in law.
\end{proof}

\begin{lma}\label{continuitylp}
Let $\mu_s=\hat P_{t,s}\mu$. Then
there exist constants $c,c'>0$ depending only on $K,N$ and $L$ such that
\begin{align*}
W_{p,t}(\mu_s,\mu_{s'})^p\leq c|s-s'|^{p/2}e^{c'|s-s'|/2}.
\end{align*}
for all $0\leq s,s'\leq t$.
\end{lma}
\begin{proof}
Assume $0<s<s'<t$. Then by $\mu_s=\hat P_{s',s}\mu_{s'}$ we estimate
\begin{align*}
W_{p,t}(\mu_s,\mu_{s'})^p\leq \int\int d_t^p(x,y)p_{s',s}(x,y)\, dm_s(y)\, d\mu_{s'}(x).
\end{align*}
By virtue of the Gaussian upper bounds (\cite[Section 4]{lierl2015}) and the Bishop Gromov volume comparison in RCD$(K,N)$ spaces
(\cite[Theorem 2.3]{sturm2006}) we obtain for $\sigma=s'-s$
and $B_t(r,x)$ denoting the ball of radius $r$ around $x$ in the metric space $(X,d_t)$
\begin{align*}
p_{s',s}(x,y)&\leq \frac C{m_t(B_t(\sqrt\sigma, x))}\cdot \exp\Big(-\frac{d^2_t(x,y)}{C\sigma}\Big)\\
A(R,x)&\leq \Big(\frac R{r}\Big)^{N-1}\cdot e^{R\sqrt{|K|(N-1)}}\cdot A(r,x)
\end{align*}
for $R\ge r$
where $A(r,x)=\partial_{r+} m_t(B_t(r,x))$ and thus (by integrating from 0 to $\sqrt\sigma$)
\begin{eqnarray*}
A(R,x)\le N\frac {R^{N-1}}{{\sigma}^{N/2}}\cdot e^{R\sqrt{|K|(N-1)}}\cdot m_t(B_t( \sqrt\sigma,x))
\end{eqnarray*}
for $R\ge \sqrt\sigma$. Then estimating further yields (with varying constants)
\begin{equation}
\begin{aligned}\label{continuitylpproof}
& \int\int d_t^p(x,y)p_{s',s}(x,y)\, dm_s(y)\, d\mu_{s'}(x)\\
 &\leq\int_X\Big[\frac C{m_t(B_t(\sqrt\sigma, x))}\cdot \int_X d^p_t(x,y)\cdot \exp\Big(-\frac{d^2_t(x,y)}{C\sigma}\Big)dm_t(y)
\Big]d\mu_{s'}(x)\\
&\leq C\sigma^{p/2}+
C\int_X \int_{\sqrt{\sigma}}^\infty R^p\cdot 
\exp\Big(-\frac{R^2}{C\sigma}\Big)
 N\frac {R^{N-1}}{{\sigma}^{N/2}}\cdot e^{R\sqrt{|K|(N-1)}}\,dR\,d\mu_{s'}(x)\\
 &\leq C\sigma^{p/2}+c''\sigma^{p/2}e^{c'\sigma/2}\leq c\sigma^{p/2}e^{c'\sigma/2}.
\end{aligned}
\end{equation}
\end{proof}

\subsubsection*{Couplings of Brownian motions}
In the remainder of this section we will study couplings of backward Brownian motions under the additional assumption that the following transport estimate holds: For every $\mu,\nu\in\mathcal P(X)$ and every $p\in[1,\infty]$ we have 
\begin{align}\label{eq:kuwa2}
W_{p,s}(\hat P_{t,s}\mu,\hat P_{t,s}\nu)\leq W_{p,t}(\mu,\nu).
\end{align}
In Section \ref{sec:transport} we show that, under additional regularity assumptions, $(X,d_t,m_t)_{t\in I}$ is a super-Ricci flow if and only if \eqref{eq:kuwa2} holds.

 Using \eqref{eq:kuwa2} we construct couplings of two backward Brownian motions $(B_s^1)_{s\leq t},\\
(B_s^2)_{s\leq t}$
on $X$ such that the 
distance $d_s$ between $B_s^1$ and $B_s^2$ does not exceed the distance $d_t$ between $B_t^1$ and $B_t^2$. 

We adapt the strategy in \cite{sturmcoup} and introduce the $\sigma$-field 
$$\mathcal B^u(X^2):=\bigcap_{\nu\in\mathcal P(X^2)}\mathcal B^\nu(X^2)$$
 of universally measurable subsets of $X^2$, i.e. the intersection
of all $\mathcal B^\nu(X^2)$, where $\nu$ runs through the set $\mathcal P(X^2)$ and where $\mathcal B^\nu(X^2)$ denotes the completion of the Borel $\sigma$-field on $X^2$ w.r.t.
$\nu\in\mathcal P(X^2)$. Let $\mathrm D:=\{k2^{-n}|k,n\in\mathbb N\}\cap (0,t]$ denote the set of nonnegative dyadic number s in $(0,t]$ and $\mathrm D_n:=
\{k2^{-n}|k\in\mathbb N\}\cap (0,t]$ for fixed $n\in\mathbb N$.

\begin{lma}
Assume that \eqref{eq:kuwa2} holds. Then
for each $s\leq t$
there exists a Markov kernel $q_{t,s}^*$ on $(X^2,\mathcal B^u(X^2))$ with the following properties:
\begin{enumerate}
\item[i)]
For each $(x,y)\in X^2$ the probability measure $q_{t,s}^*((x,y),\cdot)$ is a coupling of the probability measures $p_{t,s}(x,\cdot)$ and $p_{t,s}(y,\cdot)$.
\item[ii)]
For each $(x,y)\in X^2$ and $q_{t,s}^*((x,y),\cdot)$-a.e. $(x',y')\in X^2$
\begin{align*}
d_s(x',y')\leq d_t(x,y).
 \end{align*} 
 \end{enumerate}
\end{lma}

\begin{proof}
By virtue of the transport estimate \eqref{eq:kuwa2} there exists at least one probability measures with properties \textit{i)} and \textit{ii)}.
Indeed, define $\mu_s=\hat P_{t,s}\delta_x$, $\nu_s=\hat P_{t,s}\delta_y$ and let 
$\gamma_p\in\Pi(\mu_s,\nu_s)$ such that $W_{p,s}(\mu_s,\nu_s)=||d_s||_{L^p(\gamma_p)}$. Since $\gamma_p\in\Pi(\mu_s,\nu_s)$, $(\gamma_p)_{p\in\mathbb N}$ is 
tight
(\cite[Lemma 4.4]{villani2009}) and hence there exists a subsequence $p_k$ and a probability measure $\gamma$ such that $\gamma_{p_k}$
weakly converges to $\gamma$. Since $\Pi(\mu_s,\nu_s)$ is closed we obtain that $\gamma\in\Pi(\mu_s,\nu_s)$. Moreover, since $d_s\wedge R\in\mathcal C_b(X\times X)$
\begin{align*}
 ||d_s\wedge R||_{L^p(\gamma)}=\lim_{k\to\infty}||d_s\wedge R||_{L^p(\gamma_{p_k})}\leq\lim_{k\to\infty}||d_s||_{L^{p_k}(\gamma_{p_k})}\leq d_t(x,y),
\end{align*}
where the second inequality follows from the H\"older inequality and the last from Corollary 2.15. Letting $R\to\infty$ and $p\to\infty$, we obtain
\begin{align*}
 ||d_s||_{L^\infty(\gamma)}\leq d_t(x,y).
\end{align*}
Hence the set of all these  couplings $\gamma$ is non-empty and satisfies \textit{i)} and \textit{ii)}.
Moreover, for given $x,y\in X$ this set is closed w.r.t. weak convergence in $\mathcal P(X^2)$.
According to the measurable selection theorem \cite[Theorem 6.9.2]{bogachev2007} we may choose a coupling $q_{t,s}^*((x,y),\cdot)$ such that the map
\begin{align*}
(x,y)\mapsto q_{t,s}^*((x,y),\cdot),\qquad (X^2,\mathcal B^u(X^2))\to (\mathcal P(X^2),\mathcal B(\mathcal P(X^2)))
\end{align*}
is measurable.
\end{proof}
\begin{lma}\label{cklemma}
Assuming \eqref{eq:kuwa2},
for each $n\in\mathbb N$ and $s,s'\in\mathrm D_n$ there exists a Markov kernel $q_{s,s'}^{(n)}$ on $(X^2,\mathcal B^u(X^2))$ with the following properties:
\begin{enumerate}
\item[i)]
For each $(x,y)\in X^2$ the probability measure $q_{s,s'}^{(n)}((x,y),\cdot)$ is a coupling of $p_{s,s'}(x,\cdot)$ and $p_{s,s'}(y,\cdot)$.
\item[ii)]
For each  $(x,y)\in X^2$
\begin{align*}
d_{s'}(x',y')\leq d_s(x,y)
 \end{align*}
for $q_{s,s'}^{(n)}((x,y),\cdot)$-a.e. $(x',y')$.
\end{enumerate}
\end{lma}
\begin{proof}
For $s=l2^{-n}$ and $s'=k2^{-n}$ with $l\geq k$ we put
\begin{align*}
q_{s,s'}^{(n)}:=q^*_{(k+1)2^{-n},s'}\circ\ldots\circ q^*_{s,(l-1)2^{-n}}.
\end{align*}
Obviously we have for $r\in \mathrm D_n$ such that $s'\leq r\leq s$, 
\begin{align}\label{ck}
q_{r,s'}^{(n)}\circ q_{s,r}^{(n)}=q_{s,s'}^{(n)}
\end{align}
 and the properties \textit{i)} and \textit{ii)} hold by iteration, cf. Lemma 2.3 in \cite{sturmcoup}.
\end{proof}

We fix a distribution $\nu\in\mathcal P(X^2)$ with marginals $\nu_1$ and $\nu_2$. Similarly as before for any finite subset $J=\{t_1,\ldots,t_r\}$ of 
$\mathrm D_n$
we consider the finite-dimensional distribution $Q_J^{(n)}$ on $(X^2)^{|J|}$ 
\begin{align*}
&Q_J^{(n)}(A_r\times\ldots\times A_1)\\
=&\int_{X^2}\int_{A_r}\ldots\int_{A_1}q^{(n)}_{t_2,t_1}((x_{2},y_{2}),d(x_1,y_1))\ldots q_{t,t_r}^{(n)}((x,y),d(x_r,y_r))\nu(d(x,y)),
\end{align*}
where $q_{t,t_r}^*=q_{l2^{-n},t_r}^{(n)}\circ q^*_{t,l2^{-n}}$ whenever $l2^{-n}<t<(l+1)2^{-n}$.
\begin{lma}\label{tightness}
For fixed finite $J\subset \mathrm D_m$ the family $\{Q_J^{(n)}|n\in\mathbb R,n\geq m\}$ is a tight family of probability measures on $(X^2)^{|J|}$.
\end{lma}
\begin{proof}
Let $J=\{t_1,\ldots,t_r\}$ with each $t_i\in \mathrm D_m$. The families $\{\hat P_{t,t_{i}}(\nu_1)|i=1,\ldots,r\}$ and $\{\hat P_{t,t_{i}}(\nu_2)|i=1,\ldots,r\}$ are tight by virtue of Prokhorov's theorem, see e.g.
\cite{bogachev2007}.
 This means that given $\varepsilon>0$ there exist compact sets $B_1,B_2\subset X$ such that for all $i=1,\ldots ,r$
\begin{align*}
\hat P_{t,t_{i}}(\nu_1)(X\setminus B_1)<\varepsilon,\quad \hat P_{t,t_{i}}(\nu_2)(X\setminus B_2)<\varepsilon.
\end{align*}
Applying $A_1\times A_2\subset X\times A_2\cup A_1\times X$ and \eqref{ck} yields for the compact set  $\vec{B}=(B_1\times B_2)^r$ and $n\in\mathbb N$
\begin{align*}
Q_J^{(n)}((X^2)^r\setminus \vec B)\leq &\sum_{i=1}^{r} Q_{t,t_{i}}^{(n)}(X^2\setminus B_1\times B_2)\\
\leq & \sum_{i=1}^{r} \Big[Q_{t,t_{i}}^{(n)}((X\setminus B_1)\times X)+Q_{t,t_{i}}^{(n)}((X\times(X\setminus B_2))\Big]\\
=&\sum_{i=1}^{r}\Big [\hat P_{t,t_{i}}(\nu_1)(X\setminus B_1)+\hat P_{t,t_{i}}(\nu_2)(X\setminus B_2)\Big]\\
\leq &2r\varepsilon,
\end{align*}
where the last two inequalities follow from \textit{i)} of Lemma \ref{cklemma} and the tightness of 
$\{\hat P_{t,t_{i}}(\nu_j)\}_i$ respectively. Hence the family $\{Q_J^{(n)}|n\in\mathbb R,n\geq m\}$ is tight.
\end{proof}

For $J=\{t_1,\ldots,t_r\}$ as above we set
\begin{align*}
\vec{e_1}\colon (X^2)^r\to X^r,\quad ((x_1,y_1),\ldots,(x_r,y_r))\mapsto (x_1,\ldots, x_r),
\end{align*}
and similarly for $\vec{e_2}$.
\begin{Prop}\label{weak}
There exists a projective family $\{Q_J^\nu| J\text{ finite }\subset\mathrm D\}$ of probability measures and a subsequence $(n_l)_{l\in\mathbb N}$ such that for each finite $J\subset\mathrm D$
\begin{enumerate}
\item[i)]
 $Q_J^{(n_l)}\to Q_J^\nu$ weakly in $\mathcal P((X^2)^{|J|})$ as $l\to\infty$,
\item[ii)] and  $(\vec{e_1})_\#Q_J^\nu=P_J^{\nu_1}$, $(\vec{e_2})_\#Q_J^\nu=P_J^{\nu_2}$.
\end{enumerate}
In particular there exists a probability measure $Q_\mathrm{D}^\nu\in \mathcal P((X^2)^\mathrm D)$ such that for all finite $J\subset\mathrm D$
\begin{align*}
(\pi_J)_\#Q_\mathrm D^\nu = Q_J^\nu
 \end{align*} 
 and 
 \begin{align*}
 (\vec{e_1})_\#Q_\mathrm D^\nu=P_\mathrm D^{\nu_1},\quad (\vec{e_2})_\#Q_\mathrm D^\nu=P_\mathrm D^{\nu_2}.
 \end{align*}
\end{Prop}
\begin{proof}
Lemma \ref{tightness} yields for each fixed $J$ the existence of a weakly converging subsequence $Q_J^{(n_l)}$ by virtue of Prokhorov's theorem. By a diagonal argument we may choose a subsequence such that $Q_J^{(n_l)}$ weakly converges for all finite $J\subset\mathrm  D$. Note that 
\begin{align*}
 (\vec{e_1})_\#Q_J^{(n_l)}=P_J^{\nu_1},\quad (\vec{e_2})_\#Q_J^{(n_l)}=P_J^{\nu_2}
\end{align*}
and hence the same holds true for the limit. We obtain the last assertion by applying Kolmogorov's extension theorem.
\end{proof}
The next theorem is in particular true for super-Ricci flows satisfying additionally \eqref{assumption3} and \eqref{assumption4}. This is summarized in Theorem \ref{thm3} which we prove in Section \ref{sec:proofofeq}.
\begin{thm}\label{brownian}
Let $(X,d_t,m_t)_{t\in I}$ be a family of RCD$(K,N)$ spaces such that \eqref{assumption1} and \eqref{assumption2} hold. Moreover we assume that the transport estimate
\eqref{eq:kuwa2} holds for every $p\in[1,\infty]$. Then,
for each $x,y\in X$ there exists a continuous stochastic process $(B_s)_{s\leq t}$ such that $(B_s)_{s\leq t}$ is a coupling of the backward Brownian motions 
$(B_s^1)_{s\leq t}$ and $(B_s^2)_{s\leq t}$ with values in $X$ and terminal distributions $\delta_x$ and $\delta_y$ respectively and it satisfies 
for $Q_\mathrm D^{(\delta_x,\delta_y)}$-a.e. path
\begin{align*}
d_{s}(B_{s}^1,B_{s}^2)\leq d_{t}(x,y),
\end{align*}
for each $s\leq t$.
\end{thm}
\begin{proof}
Set $\nu=(\nu_1,\nu_2)=(\delta_x,\delta_y)$.
Consider the coordinate process $\pi_s=(\pi_s^1,\pi_s^2)\colon (X^2)^\mathrm D\to X^2$. 
Under $Q_\mathrm D^\nu$ the process $(\pi_s^1)_{s\in\mathrm D}$ has distribution 
$P_\mathrm D^{\nu_1}$ and satisfies the continuity property \eqref{eq:continuitylp}. The corresponding statement holds true for the process $(\pi_s^2)_{s\in\mathrm D}$. 
Hence, the process $\pi_t=(\pi_s^1,\pi_s^2)$ satisfies the Kolmogorov continuity theorem for $\alpha>2$
since
\begin{align*}
 E[\hat d_t(\pi_s,\pi_{s'})^\alpha]\leq &2^{\alpha/2}\Big(E[d_t(\pi_s^1,\pi_{s'}^1)^{\alpha}]+E[d_t(\pi_s^2,\pi_{s'}^2)^{\alpha}]\Big)\\
 \leq &c2^{\alpha/2}|s-s'|^{\alpha/2},
\end{align*}
with product metric $\hat d^2((x^1,y^1),(x^2,y^2))=d^2(x^1,x^2)+d^2(y^1,y^2)$.
Consequently there exists a continuous 
modification $(B_s)_{s\leq t}=(B_s^1,B_s^2)_{s\leq t}$ defined by
$B_s=\lim_{s'\to s,s\in\mathrm D}\pi_{s'}$
 for $Q_\mathrm D^\nu$-a.e. $\omega$ and all $s\leq t$, cf. Lemma 63.5 in \cite{bauerprob}. The process $(B_s^i)_{s\leq t}$, $i=1,2$ is a backward Brownian motion
 by continuity of $s\mapsto p_{t,s}(x,dy)$.
 
We need to justify that for $Q_\mathrm D^\nu$-a.e. path
\begin{align*}
d_{s}(B_{s}^1,B_{s}^2)\leq d_{t}(x,y).
\end{align*}
For each $n\in\mathbb N$ let $Q_{\mathrm D_n}^{(n)}$ be the projective limit of the family $(Q_J^{(n)})_{J\subset \mathrm D_n}$, which exists thanks to the
Kolmogorov extension theorem.
Consider the coordinate process $(\pi_s^{(n)})_{s\in\mathrm D_n}=(\pi_s^{1,(n)},\pi_s^{2,(n)})_{s\in\mathrm D_n}$
from $(X^2)^{\mathrm D_n}\to X^2$. Then $Q_{\mathrm D_n}^{(n)}$-a.e. we have $d_s(\pi_s^{1,(n)},\pi_s^{2,(n)})\leq d(x,y)$ by virtue of Lemma \ref{cklemma}. Applying Proposition
\ref{weak} and \textit{ii)} of Lemma \ref{cklemma} we obtain for a subsequence
\begin{align*}
 &E\big[(d_s(\pi_s^1,\pi_s^2)\wedge R)^p\big]^{1/p}=\lim_{l\to\infty}E\big[(d_s(\pi_s^{1,(n_l)},\pi_s^{2,(n_l)})\wedge R)^p\big]^{1/p}\\
 &\leq\lim_{l\to\infty} E\big[(d_t(x,y)\wedge R)^p\big]^{1/p}=d_t(x,y)\wedge R,
\end{align*}
for each $s\in\mathrm D$.
Letting $R$ and $p$ tend to $\infty$ we find for each $s\in\mathrm D$
\begin{align*}
 d_s(\pi_s^1,\pi_s^2)\leq d_t(x,y).
\end{align*}
Since the process $(B_s)_{s\in \mathrm D}$ is a modification we get for each $s\in\mathrm D$ and $Q_\mathrm D^\nu$-a.e.
$ d_s(B_s^1,B_s^2)\leq d_t(x,y)$. Since $\mathrm D\subset (0,t]$ is a dense and countable subset we obtain the result by continuity of $s\mapsto B_s(\omega)$.

\end{proof}

  \section{New characterization of super-Ricci flows}\label{sec:transport}
Let $(X,d_t,m_t)_{t\in I}$ be as in Section \ref{sec:proofofmain}. The main task of this section is to show that \eqref{eq:kuwa2} holds if and only if $(X,d_t,m_t)_{t\in I}$ is a super-Ricci flow.  From Definition \ref{supereva} we know already that the latter is equivalent to \eqref{eq:kuwa2} for $p=2$.
So we still need to show that \eqref{eq:kuwa2} holds for $p>2$. Crucial for this is to establish a stronger dynamic Bochner inequality.

  For this we afford more regularity of the map $r\mapsto \log d_r(x,y)$. We assume that there exists a  $\mathcal C^0$ map $r\mapsto h_r(x,y)$, 
  uniformly bounded $|h_r(x,y)|\leq C$ such that for each $s,t\in I$ and $x,y\in X$
  \begin{align}\label{assumption3}
   d_t(x,y)=d_s(x,y)e^{\int_s^th_r(x,y)\, dr}.
  \end{align}
  Consequently, for each $x,y\in X$, $r\mapsto \log d_r(x,y)$ is continuously differentiable with derivative
$h_r(x,y)=\frac{d}{dr}\log d_r(x,y)$.\\
\noindent
Moreover we assume that 
\begin{equation}
 \begin{aligned}\label{assumption4}
 &\forall x\in X, r\in I\text{ the limit }\lim_{y\to x}h_r(x,y):=H_r(x)\text{ exists, measurable in }x,\\
 &\text{and }r\mapsto H_r(x) \text{ is continuous }\forall x\in X.
\end{aligned}
\end{equation}

 We obtain the following lemma.
\begin{lma}
 Let $u\in\Lip(X)$. Then, assuming \eqref{assumption3} and \eqref{assumption4}, for all $s,t\in I$ and $x\in X$
 \begin{align*}
  \lip_t u(x)=\lip_s u(x) e^{-\int_s^t H_r(x)\, dr}.
 \end{align*}

\end{lma}

\begin{proof}
For $s<t$, we obtain from the very definition of the local slope
 \begin{align*}
  \lip_tu(x)=&\limsup_{y\to x}\frac{|u(y)-u(x)|}{d_t(x,y)}\leq\limsup_{y\to x}\frac{|u(y)-u(x)|}{d_s(x,y)}e^{-\liminf_{y\to x}\int_s^th_r(x,y)\, dr}\\
  =&\lip_su(x)e^{-\int_s^tH_r(x)\, dr},
 \end{align*}
where we applied dominated convergence.
Changing the roles of $s$ and $t$ yields
\begin{align*}
  \lip_su(x)=&\limsup_{y\to x}\frac{|u(y)-u(x)|}{d_s(x,y)}\leq\limsup_{y\to x}\frac{|u(y)-u(x)|}{d_t(x,y)}e^{-\liminf_{y\to x}\int_s^th_r(x,y)\, dr}\\
  =&\lip_tu(x)e^{-\int_s^tH_r(x)\, dr},
 \end{align*}
 which proves the assertion.
\end{proof}
We apply our observation to the minimal relaxed gradient. We say that $G\in L^2(X,m_t)$ is a \emph{$t$-relaxed gradient} of $u\in L^2(X,m_t)$ if there exists Lipschitz functions
$u_n\in L^2(X,m_t)$ such that
\begin{align*}
 u_n\to u\text{ in }L^2(X,m_t)\text{ and }\lip_tu_n\rightharpoonup\tilde G \text{ in }L^2(X,m_t),\, \tilde G\leq G\, m\text{-a.e. in }X. 
\end{align*}
$G$ is the \emph{minimal $t$-relaxed gradient} $|\nabla_tu|_*$ if its $L^2(X,m_t)$ norm is minimal among all relaxed gradients, see \cite[Definition 4.2]{agscalc}.
The collection of all $t$-relaxed gradients is convex and closed in $L^2(X,m_t)$ \cite[Lemma 4.3]{agscalc}.
\begin{Prop}\label{propdiffbar}
Assume \eqref{assumption3} and \eqref{assumption4} hold.
Then, for $m$-a.e. $x\in X$, we have
 \begin{align*}
  |\nabla_tu|_*(x)=|\nabla_su|_*(x)e^{-\int_s^tH_r(x)\, dr}
 \end{align*}
for each $u\in\F$ and for all $s, t\in I$.
In particular for $m$-a.e. $x\in X$, $t\mapsto|\nabla_tu|_*(x)$ is continuously differentiable.
\end{Prop}
\begin{proof}
Assume $s\leq t$.
 Let $u_n\in L^2(X,m_s)$ be a sequence of Borel Lipschitz functions such that $u_n\to u$ and $\lip_su_n\to |\nabla_s u|_*$ in $L^2(X,m_s)$, see Lemma 4.3 in
 \cite{agscalc}.
 Then 
 since $H$ is uniformly bounded
 \begin{align*}
  \lip_su_n(\cdot)e^{-\int_s^tH_r(\cdot)\, dr}\to |\nabla_s u|_*(\cdot)e^{-\int_s^tH_r(\cdot)\, dr} \text{ in }L^2(X,m_s).
 \end{align*}
This implies that $|\nabla_s u|_*(\cdot)e^{-\int_s^tH_r(\cdot)\, dr}$ is a relaxed gradient of $u$ with respect to the $d_t$ norm, and hence from Lemma 4.4 
in \cite{agscalc}
\begin{align*}
 |\nabla_tu|_*(\cdot)\leq |\nabla_su|_*(\cdot)e^{-\int_s^tH_r(\cdot)\, dr}\quad m\text{- a.e. in }X.
\end{align*}
Changing the roles of $s$ and $t$ yields that 
\begin{align*}
 |\nabla_tu|_*(\cdot)= |\nabla_su|_*(\cdot)e^{-\int_s^tH_r(\cdot)\, dr}\quad m\text{- a.e. in }X.
\end{align*}
Choosing $s$ and $t$ from a dense and countable set $D$ in $I$ the argument from above implies that $m$-a.e. in X
\begin{align}\label{dasteil}
 |\nabla_{t}u|_*(\cdot)= |\nabla_{s}u|_*(\cdot)e^{-\int_{s}^{t}H_r(\cdot)\, dr}
\end{align}
for each $s$ and $t$ in $D$.
Since the dependence of the left and the right side of the equality is continuous with respect to $s$ and $t$,
 we conclude that for $m$-a.e. $x\in X$, $|\nabla_{t}u|_*(\cdot)= |\nabla_{s}u|_*(\cdot)e^{-\int_{s}^{t}H_r(\cdot)\, dr}$
holds for every $s$ and $t$ in $I$. 

Similarly, we choose $u$ in a dense and countable set $C$ in $\F$ (\cite[Proposition 4.10]{agsmet}) and obtain that $m$-a.e.
equation \eqref{dasteil} holds for every $s,t\in I$ and every $u\in C$. Given $u\in\F$ we approximate $u$ by a sequence $u_n\in C$, i.e.
$|\nabla_tu_n|\to|\nabla_t u|$ in $L^2(X,m_t)$. Then there exists a subsequence $u_{n_k}$ such that for $m$-a.e. $x\in X$, $|\nabla_tu_{n_k}|(x)\to|\nabla_t u|(x) $. 
Equality \eqref{dasteil} implies that for the same subsequence $|\nabla_su_{n_k}|(x)\to|\nabla_s u|(x) $ for $m$-a.e. $x$. Hence we showed that for $m$-a.e. $x\in X$,
\eqref{dasteil} holds for every $u\in\F$ and every $s,t\in I$.

The last assertion follows directly from the fact that $r\mapsto H_r(x)$ is supposed to be continuous for all $x\in X$.

\end{proof}

We will now state the main theorem of this section.

 \begin{thm}\label{thm3}
Let $(X,d_t,m_t)_{t\in I}$ be a one-parameter family of geodesic Polish metric measure spaces satisfying \eqref{assumption1}, \eqref{assumption2}, \eqref{assumption3} and \eqref{assumption4} such that each $(X,d_t,m_t)$ is a RCD$(K,N)$ space.. Then, $(X,d_t,m_t)_{t\in I}$ is a super-Ricci flow if and only if one of the following equivalent properties holds
\begin{enumerate}
\item[i)] for all $s\leq t$, $\alpha\in[1/2,1]$ and $u\in \Dom(\Ch)$
 \begin{align*}
 (\Gamma_t(P_{t,s}u))^\alpha\leq P_{t,s}(\Gamma_s(u)^\alpha)\quad m\text{-a.e.},
     \end{align*}  
\item[ii)] for all $s\leq t$, $p\in[1,\infty]$ and $\mu,\nu\in\mathcal P(X)$
\begin{align*}
W_{p,s}(\hat P_{t,s}\mu,\hat P_{t,s}\nu)\leq W_{p,t}(\mu,\nu),
\end{align*}
\item[iii)] there exists a coupling of backward Brownian motions
$(B_s^1,B_s^2)_{s\leq t}$ terminating in $x$ and $y$ respectively such that for all $s\leq t$
\begin{align*}
 d_{s}(B_{s}^1,B_{s}^2)\leq d_t(x,y) \text{ almost surely}.
\end{align*}
\end{enumerate}
\end{thm}
The proof of this is given in Section \ref{sec:proofofeq}. In the following two subsections we prepare the necessary ingredients for the proof.

  \subsection{Dynamic Bochner inequality and gradient estimates}
As stated in Definition \ref{supereva}, the dynamic Bochner inequality \eqref{est-IVeva} is equivalent to the $L^2$-gradient estimate \eqref{est-IIIeva}. In the static case (``$\partial_t\Gamma_t=0$") it is well-known \cite{savare, bakrytrans} that it is also equivalent to the stronger $L^1$-gradient estimate, which yields stronger transport estimates and contraction estimates of couplings of Brownian motions \cite{sturmrenesse, sturmcoup}.

The aim in this section is to prove a time-dependent version of the $L^1$-gradient estimate
\begin{align}\label{eq:kuwa2a}
|\nabla_tP_{t,s}u|_*\leq P_{t,s}(|\nabla_s u|_*),
\end{align}
which in return will imply the stronger transport estimates \eqref{eq:kuwa2} by Kuwada's duality. For this we give a new, more appealing definition of a time-dependent Bochner inequality.
\begin{Def}
 We say that the \emph{pointwise dynamic Bochner inequality} holds at time $t$ if for all $u\in\Dom(\Delta_t)\cap L^\infty(X,m_t)$ such that $\Gamma_t(u)\in L^\infty(X,m_t)$, and all 
 $g\in\Dom(\Delta_t)\cap L^\infty(X,m_t)$ with $g\geq 0$
 \begin{align}\label{Bochnereva}
 \frac12 \int\Gamma_t(u)\Delta_tg\, dm_t+\int(\Delta_tu)^2g+\Gamma_t(u,g)\Delta_tu\, dm_t\geq\frac12\int(\partial_t\Gamma_t)(u)g\, dm_t.
 \end{align}

\end{Def}
This is a ``real'' Bochner inequality in the sense that on the one hand $u$ and $g$ do not have to arise as a heat flow 
(see \cite[Definition 5.5]{sturm2016}), and on the other we employ the time-derivative
$\partial_t\Gamma_t(u)$ in contrast to \cite[Definition 5.5, Definition 5.6]{sturm2016}). We will show in Section \ref{sec:bochner} that \eqref{Bochnereva} holds if and only if $(X,d_t,m_t)_{t\in I}$ is a super-Ricci flow, provided \eqref{assumption3} and \eqref{assumption4} hold. Under sufficient regularity of $u$ ($\Delta_t u\in\Dom(\E)$, $\Gamma_t (u)\in \Dom(\Delta_t)$), the left hand side can be expressed with the usual $\Gamma_{2,t}(u)=\frac12\Delta_t\Gamma_t(u)-\Gamma_t(u,\Delta_t u)$, but for the purpose of Section \ref{sec:bochner}, we deal with less regularity.

In the following theorem we assume that
\begin{align}\label{reg-bocheva}
 u_r\in\Lip(X)\text{ for all }r\in  (s,t)\text{ with }\sup_{r,x}\lip_ru_r(x)<\infty,
\end{align}
where $u_r=P_{r,s}u$ and $u$ is some Lipschitz function. This is not a restriction since the $L^2$-gradient estimate \eqref{est-IIIeva} in Definition \ref{supereva} will imply this.
 \begin{thm}\label{thm1}
Let $(X,d_t,m_t)_{t\in I}$ be a one-parameter family of geodesic Polish metric measure spaces satisfying \eqref{assumption1}, \eqref{assumption2}, \eqref{assumption3}
and \eqref{assumption4} such that each $(X,d_t,m_t)$ is a RCD$(K,N)$ space. Then, if the pointwise dynamic Bochner inequality \eqref{Bochnereva} for each $s\leq r\leq t$ and the regularity assumption \eqref{reg-bocheva} is satisfied,
for every $\alpha\in[1/2,1]$ we have for a.e.\ $\tau\leq t$ and $\sigma\geq s$ and every $u\in\Dom(\Ch)$
 \begin{align*}
( \Gamma_\tau(P_{\tau,\sigma}u))^\alpha\leq P_{\tau,\sigma}(\Gamma_\sigma(u)^\alpha) \quad m\text{-a.e..}
     \end{align*} 
\end{thm}

In order to prove the theorem we adapt the strategy in \cite{savare}, but we have to take care of the new term on the 
right hand side of \eqref{Bochnereva} and of the different domains in \eqref{Bochnereva} compared to \cite{savare}.
  \subsubsection*{Quasi-regular Dirichlet forms}
  We briefly recall the notion of quasi-regular Dirichlet forms developed in \cite{Ma} and \cite{fukushima2012}.
  We denote by $\F=\{u\in L^2(X,m)|\E(u)<\infty\}$ the domain of a Dirichlet form $\E\colon L^2(X,m)\to [0,\infty]$, where $X$ is a Polish space and $m$ is a $\sigma$-finite Borel measure. $\F$ is a Hilbert space with norm $||u||_\F^2=||u||^2_{L^2(X,m)}+\E(u)$. If $F$ is a closed set in $X$ we denote
  \begin{align*}
  \F_F:=\{u\in\F|u(x)=0\text{ for }m\text{-a.e. }x\in X\setminus F\}.
  \end{align*}
  \begin{Def}[\cite{Ma},\cite{fukushima2012}]
  Given a Dirichlet form $\E$ on a Polish space $X$, an $\E$-nest is an increasing sequence of closed subsets $(F_k)_{k\in\mathbb N}\subset X$ such that $\cup_{k\in\mathbb N}\mathcal F_{F_k}$ is dense in $\F$.\\
  A set $N\subset X$ is $\E$-polar if there is an $\E$-nest $(F_k)_{k\in\mathbb N}$ such that $N\subset X\setminus \cup_{k\in\mathbb N}F_k$. If a property holds  in a complement of an $\E$-polar set we say that it holds $\E$-quasi-everywhere ($\E$-q.e.).\\
  A function $u\colon X\to\mathbb R$ is said to be $\E$-quasi-continuous if there exists an $\E$-nest $(F_k)_{k\in\mathbb N}$ such that every restriction $f_{|F_k}$ is continuous on $F_k$.\\
  
  \noindent
  The Dirichlet form $\E$ is said to be  quasi-regular if the following three properties hold.
  \begin{enumerate}
  \item[i)] There exists an $\E$-nest $(F_k)_{k\in\mathbb N}$ consisting of compact sets.
  \item[ii)] There exists a dense subset of $\F$ whose elements have $\mathcal E$-quasi-continuous representatives.
  \item[iii)] There exists an $\E$-polar set $N\subset X$ and a countable collection of $\E$-quasi-continuous functions $(f_k)_{k\in\mathbb N}\subset \F$ separating the points of $X\setminus N$.
  \end{enumerate}
 
  \end{Def}
  For every $u\in \F$ the quasi-regularity implies that $u$ admits an $\E$-quasi-continuous representative $\tilde u$. The representative is 
  unique $q.e.$ and
  \begin{align}\label{quasi-continuous}
   \text{if }u\in\F\text{ with }|u|\leq C\, m\text{-a.e., then }|\tilde u|\leq C\text{ q.e..}
  \end{align}

  The following Lemma is taken from \cite{savare}.
  \begin{lma}[{\cite[Lemma 2.6]{savare}}]\label{sstuff}
  Let $\E$ be a strongly local, quasi-regular Dirichlet form with linear generator $\Delta$.
   Let $\psi\in L^1(X,m)\cap L^\infty(X,m)$ nonnegative and $\varphi\in L^1(X,m)\cap L^2(X,m)$ such that
   \begin{align*}
    \int_X \psi\Delta g\, dm\geq-\int_X\varphi g\, dm
   \end{align*}
for any nonnegative $g\in \F\cap L^\infty(X,m)$ with $\Delta g\in L^\infty(X,m)$. Then $\psi\in \F$ with
\begin{align*}
 \mathcal E(\psi)\leq\int_X\psi\varphi\, dm, \quad \int\varphi\, dm\geq0,
\end{align*}
and there exists a unique finite Borel measure $\mu:=\mu_+-\varphi m$ with $\mu_+\geq 0$, $\mu_+(X)\leq \int\varphi\, dm$ such that
every $\E$-polar set is $|\mu|$-negligible, the q.c. representative of any function in $\F$ belongs to $L^1(X,|\mu|)$ and
\begin{align*}
 -\mathcal E(\psi,g)=-\int \Gamma(\psi,g)\, dm=\int\tilde g\, d\mu\text{ for every }g\in\F.
\end{align*}

  \end{lma}

  We denote by $\Delta^*u$ the measured valued Laplacian, i.e. the signed measure $\mu=\mu_+-\mu_-$ such that
  \begin{align}\label{greatstuff}
   \mathcal E(u,\varphi)=\int\tilde\varphi\, d\mu \text{ for every }\varphi\in\F.
  \end{align}
\subsubsection*{Contraction estimates for the heat flows $P_{t,s}$ and $\hat P_{t,s}$}
Recall that on a family of closed Riemannian manifolds $(M,g_t)$ we obtain the equality
\begin{align*}
 H_t[u](g,h)=\langle\nabla_t^2u\nabla_t g,\nabla_t h\rangle_{g_t}.
\end{align*}
Further note that $|\langle\nabla_t^2u\nabla_t g,\nabla_t h\rangle_{g_t}|\leq |\nabla^2_tu|_{\textrm{HS}}|\nabla_t g||\nabla_t h|$, 
where $|\cdot|_{\textrm{HS}}$ denotes the Hilbert-Schmidt norm.
If the manifold has Ricci curvature bounded from below by some $K\in\mathbb R$ then with $||\cdot||_2=||\cdot||_{L^2}$ and $K_-=\max\{-K,0\}$
\begin{align*}
 |||\nabla_t^2u|_{\textrm{HS}}||_2^2\leq 2||\Delta_tu||_2^2+2K_-\E_t(u),
\end{align*}
where we used the static Bochner inequality and integration by parts.

For each $t\in I$ we define the ``Hessian"
\begin{align*}
 H_t[u](g,h):=\frac12\Big(\Gamma_t(g,\Gamma_t(u,h))+\Gamma_t(h,\Gamma_t(u,g))-\Gamma_t(u,\Gamma_t(g,h))\Big).
\end{align*}
Moreover, we define the distribution valued $\Gamma_2$-operator
\begin{align*}
 \Gamma_{2,t}(u)\colon\mathcal F\cap L^\infty\cap L^1\to\mathbb R
\end{align*}
as in \cite{sturm2016}.
\begin{Def}
 For each $u\in\Dom(\Delta_t)$ such that $u,\Gamma_t(u)\in L^\infty(X,m_t)$ we define
 \begin{align*}
  \Gamma_{2,t}(u)(g)=\int-\frac12\Gamma_t(\Gamma_t(u),g)\, dm_t+\int(g(\Delta_tu)^2+\Gamma_t(g,u)\Delta_tu)\, dm_t,
 \end{align*}
where $g\in\F$ such that $g\in L^1(X,m_t)\cap L^\infty(X,m_t)$.
\end{Def}
Note that thanks to the static RCD$(K,\infty)$-condition the domain of the Laplacian is contained in the domain of the Hessian, i.e. 
$\Dom(\Delta_t)\subset W^{2,2}(X,d_t,m_t)$, and thus $\Gamma_{2,t}(u)(g)\in\mathbb R$ for $u,g$ as above. Indeed, using \cite[Corollary 3.3.9]{gigli2014nonsmooth}, we get
\begin{align*}
\frac14\int|\Gamma_t(\Gamma_t(u),g)|\, dm_t\leq  ||\sqrt{\Gamma_t(u)}||_\infty\sqrt{\E_t(g)}(||\Delta_tu||_2+\sqrt{K_-\E_t(u)}),
\end{align*}
and thus the following estimate holds
\begin{align}\label{estimatehesslaplace}
 | \Gamma_{2,t}(u)(g)|\leq ||g||_\infty||\Delta_tu||_2^2+C||\sqrt{\Gamma_t(u)}||_\infty\sqrt{\E_t(g)}(||\Delta_tu||_2+\sqrt{K_-\E_t(u)}),
\end{align}
cf. Section 5 in \cite{sturm2016}. 
Moreover, each $\E_t=2\Ch_t$ defines a quasi-regular Dirichlet form (\cite[Theorem 4.1]{savare}).

  \begin{Prop}\label{mystuff}
  Let $(X,d_t,m_t)_{t\in I}$ satisfy the regularity assumptions \eqref{assumption3} and \eqref{assumption4}. Assume that the pointwise dynamic Bochner inequality \eqref{Bochnereva} holds at time $t\in I$. Then for every $u\in\Dom(\Delta_t)$ with $u,\Gamma_t(u)\in L^\infty(X,m_t)$
  \begin{enumerate}
 \item[i)] $\Gamma_t(u)\in\mathcal F$ with
 \begin{align*}
  \frac12\E_t(\Gamma_t(u))\leq L&||\Gamma_t(u)||_\infty\E_t(u)+||\Gamma_t(u)||_\infty||\Delta_tu||_2^2\\
  &+C||\Delta_tu||_2\sqrt{||\Gamma_t(u)^2||_\infty(||\Delta_tu||^2_2+K_-\E_t(u))}.
 \end{align*}
\item[ii)] There exists a finite nonnegative Borel measure $\mu_+=\mu_+(t)$ such that every $\E_t$-polar set is $\mu_+$-negligible and for each $g\in\mathcal F$ the $\E_t$-q.c. representative $\tilde g\in L^1(X,\mu_+)$ with
\begin{align*}
 2\Gamma_{2,t}(u)(g)=\int g(\partial_t\Gamma_t)(u)\, dm_t+\int \tilde g\, d\mu_+.
\end{align*}
In particular $\Gamma_{2,t}(u)$ is a finite Borel measure with 
\begin{align*}
 2\Gamma_{2,t}(u)=(\partial_t\Gamma_t)(u)m+\mu_+.
\end{align*}

  \end{enumerate}
  \end{Prop}
  \begin{proof}
  Let $u_\varepsilon=h_\varepsilon^tu$.
Choosing $\psi=\Gamma_t(u_\varepsilon)$ and $\varphi=-(\partial_t\Gamma_t)(u_\varepsilon)-2\Gamma_t(u_\varepsilon,\Delta_tu_\varepsilon)$ in 
Lemma \ref{sstuff} and applying the pointwise dynamic Bochner inequality together with the Leibniz rule yields
\begin{align*}
 \mathcal E_t(\Gamma_t(u_\varepsilon))\leq -\int\Gamma_t(u_\varepsilon)((\partial_t\Gamma_t)(u_\varepsilon)+2\Gamma_t(u_\varepsilon,\Delta_tu_\varepsilon))\, dm_t.
\end{align*}
Applying the Leibniz rule once again we obtain
\begin{align*}
 \mathcal E_t(\Gamma_t(u_\varepsilon))\leq -\int(\Gamma_t(u_\varepsilon)(\partial_t\Gamma_t)(u_\varepsilon)-2(\Delta_tu_\varepsilon)^2\Gamma_t(u_\varepsilon)-2\Gamma_t(u_\varepsilon,\Gamma_t(u_\varepsilon))\Delta_tu_\varepsilon)\, dm_t.
\end{align*}
Note that as $\varepsilon\to0$, $\Gamma(u_\varepsilon)\to \Gamma(u)$ pointwise, in $L^1$ and in the weak$^*$ $L^\infty$ topology. The latter is due to the fact 
that $\Gamma(u_\varepsilon-u)$ is uniformly bounded and converges to $0$ in $L^1$. Moreover by the uniform boundedness of $\Gamma(u_\varepsilon)$ in $L^\infty$
we obtain that $\Gamma(u_\varepsilon)\to \Gamma(u)$ in $L^2$. Hence we find
\begin{align*}
 \E_t(\Gamma_t(u))\leq \liminf_{\varepsilon\to0}\E_t(\Gamma_t(u_\varepsilon))
 \end{align*}
 and by Proposition \ref{propdiffbar}
 \begin{align*}
 &\int\Gamma_t(u)(\partial_t\Gamma_t)(u)\, dm_t=\int\Gamma_t(u)^2e^{H_t}\, dm_t\\
 &=\lim_{\varepsilon\to0}\int\Gamma_t(u_\varepsilon)^2e^{H_t}\, dm_t=\lim_{\varepsilon\to0}\int\Gamma_t(u_\varepsilon)(\partial_t\Gamma_t)(u_\varepsilon)\, dm_t,\\
\end{align*}
while
\begin{align*}
\int(\Delta_tu)^2\Gamma_t(u)\, dm_t=\lim_{\varepsilon\to0}\int(h_\varepsilon^t\Delta_tu)^2\Gamma_t(u_\varepsilon)\, dm_t=\lim_{\varepsilon\to0}\int(\Delta_tu_\varepsilon)^2\Gamma_t(u_\varepsilon)\, dm_t.
 \end{align*}
 In order to show that 
 \begin{align*}
  \lim_{\varepsilon\to0}\int\Gamma_t(u_\varepsilon,\Gamma_t(u_\varepsilon))\Delta_tu_\varepsilon\, dm_t=\int\Gamma_t(u,\Gamma_t(u))\Delta_tu\, dm_t,
 \end{align*}
we show that $\Gamma_t(u_\varepsilon,\Gamma_t(u_\varepsilon))$ weakly converges to $\Gamma_t(u,\Gamma_t(u))$ in $L^2$. Take a sufficiently smooth testfunction
$\varphi$ ($\varphi\in\F\cap L^\infty$), then we easily deduce
\begin{align*}
 \int\Gamma_t(u_\varepsilon,\Gamma_t(u_\varepsilon))\varphi\, dm_t=-\int\Delta_tu_\varepsilon\Gamma_t(u_\varepsilon)\varphi\, dm_t-\int\Gamma_t(u_\varepsilon,\varphi)\Gamma_t(u_\varepsilon)\, dm_t\\
 \to -\int\Delta_tu\Gamma_t(u)\varphi\, dm_t-\int\Gamma_t(u,\varphi)\Gamma_t(u)\, dm_t
\end{align*}
by the strong $L^2$ convergence of $\Delta_tu_\varepsilon$, the weak$^*$-$L^\infty$ convergence of $\Gamma(u_\varepsilon)$ and the $L^1$ convergence of $\Gamma(u_\varepsilon,\varphi)$.
Moreover $||\Gamma_t(u_\varepsilon,\Gamma_t(u_\varepsilon))||_2$ is uniformly bounded in $\varepsilon$ since
\begin{align*}
 \int|\Gamma_t(u_\varepsilon,\Gamma_t(u_\varepsilon))|^2\, dm_t\leq 4||\Gamma_t(u_\varepsilon)^2||_\infty (||\Delta_tu_\varepsilon||^2_2+K_-\E_t(u_\varepsilon))\\
 \leq C||\Gamma_t(u)^2||_\infty||(||\Delta_tu||^2_2+K_-\E_t(u))
\end{align*}
where we used \cite[Corollary 3.3.9]{gigli2014nonsmooth}. Consequently we obtain that
$\Gamma_t(u_\varepsilon,\Gamma_t(u_\varepsilon))$ weakly converges to $\Gamma_t(u,\Gamma_t(u))$ in $L^2$ since $\F\cap L^\infty$ is dense in $L^2$ \cite[Theorem 4.5]{agscalc}.

We conclude 
\begin{align*}
 \frac12\E_t(\Gamma_t(u))\leq-\int\frac12\Gamma_t(u)(\partial_t\Gamma_t)(u)-\Gamma_t(u)(\Delta_tu)^2-\Gamma_t(u,\Gamma_t(u))\Delta_tu\, dm_t\\
 \leq L||\Gamma_t(u)||_\infty\E_t(u)+||\Gamma_t(u)||_\infty||\Delta_tu||_2^2+C||\Delta_tu||_2\sqrt{||\Gamma_t(u)^2||_\infty(||\Delta_tu||^2_2+K_-\E_t(u))}.
\end{align*}

We show the second claim again by using the semigroup mollification $u_\varepsilon:=h^t_\varepsilon u$. By Lemma \ref{sstuff} we deduce that
\begin{align*}
\int g\, d\Delta_t^*\Gamma_t(u_\varepsilon)-\int \tilde g2\Gamma_t(u_\varepsilon,\Delta_tu_\varepsilon)\, dm_t\\
 =\int \tilde g\, d\mu_+(u_\varepsilon)+\int \tilde g(\partial_t\Gamma_t)(u_\varepsilon)\, dm_t,
\end{align*}
where $\Delta_t^*$ is the measure valued Laplacian, and $\mu_+(u_\varepsilon)$ the nonnegative Borel measure with 
$\mu_+(u_\varepsilon)(X)\leq \int(\Delta_tu_\varepsilon)^2+\frac12(\partial_t\Gamma_t)(u_\varepsilon)\, dm_t$. Hence, since $g=\tilde g$ q.e.
\begin{align*}
 &\int g\, d\mu_+(u_\varepsilon)\\
 =&\int-\Gamma_t(\Gamma_t(u_\varepsilon),g)\, dm_t+\int 2 g(\Delta_tu_\varepsilon)^2+2\Gamma_t( g,u_\varepsilon)(\Delta_tu_\varepsilon)\, dm_t-\int g(\partial_t\Gamma_t)(u_\varepsilon)\, dm_t.
\end{align*}
Note that the right hand side converges as $\varepsilon\to0$ since $\Gamma(u_\varepsilon)\to\Gamma(u)$ weakly in $\F$. Indeed, take a test function $\varphi\in\Dom(\Delta_t)$.
Then 
\begin{align*}
 \lim_{\varepsilon\to0}\int\Gamma_t(\Gamma_t(u_\varepsilon),\varphi)\, dm_t=-\lim_{\varepsilon\to0}\int\Gamma_t(u_\varepsilon)\Delta_t\varphi\, dm_t=\int\Gamma_t(\Gamma_t(u),\varphi)\, dm_t.
\end{align*}
Since $\E_t(\Gamma_t(u_\varepsilon))$ is uniformly bounded in $\varepsilon$ by the first claim and $\Dom(\Delta_t)$ is dense in $\F$ we deduce that
\begin{align*}
 \lim_{\varepsilon\to0}\int\Gamma_t(\Gamma_t(u_\varepsilon),g)\, dm_t=\int\Gamma_t(\Gamma_t(u),g)\, dm_t\qquad \forall \, g\in\F.
\end{align*}

Define the linear functional $\tilde\mu_+(u)\colon \F\cap L^\infty\to\mathbb R$ by
\begin{align*}
 \tilde\mu_+(u)(g):=\lim_{\varepsilon\to0} \int g\, d\mu_+(u_\varepsilon).
\end{align*}
Note that if $g\geq 0$ we have $\tilde\mu_+(u)(g)\geq0$ by the pointwise dynamic Bochner inequality. The Hahn-Banach theorem implies that there exists a linear functional
$M\colon \F\to\mathbb R$ such that $M(g)=\mu_+(u)(g)$ for all $g\in \F\cap L^\infty$ and $M(g)\geq 0$ for all $g\in\mathcal F$ such that $g\geq0$ a.e..
Moreover, if $g\in\F$ with $g\leq 1$ $m$-a.e.
\begin{align*}
M(g)=\mu_+(u)(g)=\lim_{\varepsilon\to0}\int g\, d\mu_+(u_\varepsilon)\leq \mu_+(u_\varepsilon)(X)\leq\int(\Delta_t u)^2+C\Gamma_t(u)\, dm_t.
\end{align*}

Thus by Proposition 2.5 in \cite{savare} there exists a unique finite and nonnegative Borel measure $\mu_+$ in $X$ such that every $\E_t$-polar set is $\mu_+$-negligible and
for each $g\in\mathcal F$ the $\E_t$-q.c. representative $\tilde g\in L^1(X,\mu_+)$ with
\begin{align*}
 M(g)=\int\tilde g\, d\mu_+.
\end{align*}
Consequently 
\begin{align*}
 2\Gamma_{2,t}(u)(g)=\int g(\partial_t\Gamma_t)(u)\, dm_t+\int \tilde g\, d\mu_+,
\end{align*}
and hence $\Gamma_{2,t}$ is measure valued with $2\Gamma_{2,t}(u)=(\partial_t\Gamma_t)(u)\, m_t+ \, \mu_+$.
\end{proof}

By virtue of Lebesgue's decomposition theorem we denote by $\gamma_{2,t}(u)\in L^1(X,m_t)$ the density wrt $m_t$ 
\begin{align*}
 \Gamma_{2,t}(u)=\gamma_{2,t}(u)m_t+\Gamma_{2,t}^\perp(u),\quad \Gamma_{2,t}^\perp(u)\perp m_t,
\end{align*}
and thus by Proposition \ref{mystuff}
\begin{align}\label{mystuff2}
 \gamma_{2,t}(u)\geq \frac12(\partial_t\Gamma_t)(u)\, \,  m\text{-a.e. and }\Gamma_{2,t}^\perp(u)\geq0.
\end{align}

We define for $u,h\in\Dom(\Delta_t)$ such that $\Gamma_t(u),\Gamma_t(h)\in L^\infty(X,m_t)$
\begin{align*}
\Gamma_{2,t}(u,h)(g):=\frac14\Gamma_{2,t}(u+h)(g)-\frac14\Gamma_{2,t}(u-h)(g),
\end{align*}
where $g\in\F\cap L^\infty$. Note that the right-hand side is well-defined by \eqref{estimatehesslaplace} and
\begin{align*}
\Gamma_{2,t}(u,h)(g)=\int-\frac12\Gamma_t(g,\Gamma_t(u,h))+g\Delta_tu\Delta_th+\frac12\Delta_th\Gamma_t(u,g)+\frac12\Delta_tu\Gamma_t(h,g)\, dm_t
\end{align*}
Similarly,
\begin{align*}
\gamma_{2,t}(u,h):=\frac14\gamma_{2,t}(u+h)-\frac14\gamma_{2,t}(u-h).
\end{align*}
The following Lemma is an adaptation of Lemma 3.3 in \cite{savare}.
\begin{lma}\label{fundamental}
Let  $\bar u=(u_i)_{i=1}^n$ with $u_i\in\Dom(\Delta_t)$ such that $u,\Gamma_t(u)\in L^\infty(X,m_t)$ and let $\Psi\in\mathcal C^3(\mathbb R^n)$ with $\Psi(0)=0$. Then
\begin{align*}
\Gamma_{2,t}(\Psi(\bar u))=&\sum_{i,j}\Gamma_{2,t}(u_i,u_j)(\partial_i\Psi)({\bar u})(\partial_j\Psi)({\bar u})\\
+&2\sum_{i,j,k}(\partial_i\Psi)(\bar u)(\partial_{jk}\Psi)(\bar u)H_t[u_i](u_j,u_k)\, m_t\\
+&\sum_{i,j,k,h}(\partial_{ik}\Psi)(\bar u)(\partial_{jh}\Psi)(\bar u)\Gamma_t(u_i,u_j)\Gamma_t(u_k,u_h)\, m_t.
\end{align*}
In particular $m_t$-a.e.
\begin{align*}
\gamma_{2,t}(\Psi(\bar u))=&\sum_{i,j}\gamma_{2,t}(u_i,u_j)(\partial_i\Psi)(\bar u)(\partial_j\Psi)(\bar u)\\
+&2\sum_{i,j,k}(\partial_i\Psi)(\bar u)(\partial_{jk}\Psi)(\bar u)H_t[u_i](u_j,u_k)\\
+&\sum_{i,j,k,h}(\partial_{ik}\Psi)(\bar u)(\partial_{jh}\Psi)(\bar u)\Gamma_t(u_i,u_j)\Gamma_t(u_k,u_h).
\end{align*}
\end{lma}
\begin{proof}
Note that $\Psi(\bar u)\in\Dom (\Delta_t)$ with $\Gamma_t(u)\in L^\infty$ since
\begin{align*}
\Gamma_t(\Psi(\bar u))=&\sum_{i,j}\partial_i\Psi(\bar u)\partial_j\Psi(\bar u)\Gamma_t(u_i,u_j)\in L^1\cap L^\infty,\\
\Delta_t(\Psi(\bar u))=&\sum_i\partial_i\Psi(\bar u)\Delta_tu_i+\sum_{i,j}\partial_{ij}\Psi(\bar u)\Gamma_t(u_i,u_j)\in L^2.
\end{align*}
 Thus by definition for each $g\in\F\cap L^\infty$
 \begin{align*}
 2\Gamma_{2,t}(\Psi(\bar u))(g)=\int-\Gamma_t(g,\Gamma_t(\Psi(\bar u)))+2g(\Delta_t\Psi(\bar u))^2+2\Gamma(g,\Psi(\bar u))\Delta_t\Psi(\bar u)\, dm_t.
 \end{align*} 
We calculate using the notation $\psi=\Psi(\bar u),$ $\psi_i=\partial_i\Psi(\bar u)$ and $\psi_{ij}=\partial_{ij}\Psi(\bar u)$ for the first term
 \begin{align*}
 &\int-\Gamma_t(g,\Gamma_t(\Psi(\bar u)))\, dm_t\\
 =&\sum_{i,j}\Big\{\int-\Gamma_t(g\psi_i\psi_j,\Gamma_t(u_i,u_j))\, dm\\
 &\qquad+\int g\Big(\Gamma_t(u_i,u_j)\Delta_t(\psi_i\psi_j)+2\Gamma_t(\psi_i\psi_j,\Gamma_t(u_i,u_j))\Big)\, dm_t\Big\}\\
 = &\sum_{i,j}\int-\Gamma_t(g\psi_i\psi_j,\Gamma_t(u_i,u_j))\, dm+\int 2g\Big(I+II\Big)\, dm_t,
 \end{align*}
 where
 \begin{align*}
 I=\sum_{i,j,k,h}\Gamma_t(u_i,u_j)\Big(\psi_i(\psi_{jk}\Delta_tu_k+\psi_{jkh}\Gamma_t(u_k,u_h))+\psi_{ik}\psi_{jh}\Gamma_t(u_k,u_h)\Big)
 \end{align*}
 and
 \begin{align*}
 II=\sum_{i,j,k}\psi_i\psi_{jk}\Big(\Gamma_t(u_k,\Gamma_t(u_j,u_i))+\Gamma_t(u_j,\Gamma_t(u_i,u_k))\Big).
 \end{align*}
On the other hand
 \begin{align*}
&\int 2g(\Delta_t\Psi(\bar u))^2+2\Gamma_t(g,\Psi(\bar u))\Delta_t\Psi(\bar u)\, dm_t\\
=&\sum_{i,j}2\int\Big(\Delta_tu_i\Delta_tu_jg\psi_i\psi_j+\Gamma_t(u_i,g\psi_i\psi_j)\Delta_tu_j\Big)\, dm_t\\
-&\sum_{i,j,k,h}\int 2g\Big(\psi_i\Delta_tu_k\psi_{kj}\Gamma_t(u_i,u_j)+\psi_i\Gamma_t(u_k,u_h)\psi_{khj}\Gamma_t(u_i,u_j)\\
&\qquad\qquad\qquad+\psi_i\psi_{jk}\Gamma_t(u_i,\Gamma_t(u_j,u_k))\Big)\, dm_t.
 \end{align*}
 Adding up and collecting terms yields
 \begin{align*}
  &2\Gamma_{2,t}(\Psi(\bar u))(g)\\
  =
 &\sum_{i,j}\int\Big(-\Gamma_t(g\psi_i\psi_j,\Gamma_t(u_i,u_j))+2g\psi_i\psi_j\Delta_tu_i\Delta_tu_j+2\Gamma_t(u_i,g\psi_i\psi_j)\Delta_tu_j\Big)dm_t\\
+ &\sum_{i,j,k}\int2g\psi_i\psi_{jk}\Big(\Gamma_t(u_k,\Gamma_t(u_i,u_j))+\Gamma_t(u_j,\Gamma_t(u_i,u_k))-\Gamma_t(u_i,\Gamma_t(u_j,u_k))\Big)dm_t\\
+&\sum_{i,j,k,h}\int 2g\psi_{ik}\psi_{jk}\Gamma(u_k,u_h)\Gamma(u_i,u_j)\, dm_t\\
=&2\sum_{i,j}\Gamma_{2,t}(u_i,u_j)( g \psi_i\psi_j)
+ \sum_{i,j,k}\int4g\psi_i\psi_{jk}(H_t[u_i](u_k,u_j))\, dm_t\\
&\quad+\sum_{i,j,k,h}\int 2g\psi_{ik}\psi_{jk}\Gamma_t(u_k,u_h)\Gamma_t(u_i,u_j)\, dm_t
 \end{align*}
 for each $g\in\F\cap L^\infty$.
 
 For arbitrary $g\in\F$, set $g^n:=g\wedge n$. Then, by dominated convergence (recall that $\tilde g\in L^1(X,\mu_+)$)
 \begin{align*}
  \lim_{n\to\infty}&\int g^n\, d\Gamma_{2,t}(\Psi(\bar u))
  =\lim_{n\to\infty}\Big(\int g^n(\partial_t\Gamma_t)(\Psi(\bar u))\, dm_t+\int \tilde{g^n}\, d\mu_+\Big)\\
  &=\int g\, d\Gamma_{2,t}(\Psi(\bar u)).
 \end{align*}
Similarly we can pass to the limit for the other integrals and obtain for all $g\in\F$
 \begin{align*}
  &2\Gamma_{2,t}(\Psi(\bar u))(g)
  =2\sum_{i,j}\Gamma_{2,t}(u_i,u_j)( g \psi_i\psi_j)
+ \sum_{i,j,k}\int4g\psi_i\psi_{jk}(H_t[u_i](u_k,u_j))\, dm_t\\
&\quad+\sum_{i,j,k,h}\int 2g\psi_{ik}\psi_{jk}\Gamma_t(u_k,u_h)\Gamma_t(u_i,u_j)\, dm_t,
 \end{align*}
 and hence the result.
\end{proof}

  \begin{Prop}\label{gammaestimate}
  Let $(X,d_t,m_t)_{t\in I}$ satisfy the regularity assumptions \eqref{assumption3} and \eqref{assumption4}. Assume that the pointwise dynamic Bochner inequality \eqref{Bochnereva} holds at time $t\in I$.
 Then for every $u\in\Dom(\Delta_t)\cap L^\infty(X,m_t)$ such that $\Gamma_t(u)\in L^\infty(X,m_t)$
   \begin{align*}
    \Gamma_t(\Gamma_t(u))\leq 4\big(\gamma_{2,t}(u)-\frac12\partial_t\Gamma_t(u)\big)\Gamma_t(u).
   \end{align*} 

  \end{Prop}
  
\begin{proof}
We choose the same polynomial $\Psi\colon \mathbb R^3\to \mathbb R$ as in \cite{savare} by
\begin{align*}
 \Psi(\bar u):=\lambda u_1+(u_2-a)(u_3-b)-ab, \qquad \lambda, a, b\in\mathbb R,
\end{align*}
where $\bar u=(u_1,u_2,u_3)$, where each $u_i\in\Dom(\Delta_t)\cap L^\infty(X,m_t)$ with $\Gamma_t(u_i)\in L^\infty(X,m_t)$.
We apply Proposition \ref{mystuff} and obtain
\begin{align}\label{poly}
 \gamma_{2,t}(\Psi(\bar u))\geq\frac12(\partial_t\Gamma_t)(\Psi(\bar u))\quad m\text{-a.e. in } X,
\end{align}
where both sides of the inequality depend on $\lambda,a,b\in\mathbb R$. Choosing $\lambda,a,b$ in a dense and countable subset $D$ of $\mathbb R$ yields that
\eqref{poly} holds $m$-a.e. for all $\lambda, a, b$ in $D$. Since
\begin{align*}
 (\partial_t\Gamma_t)(\Psi(\bar u))=\sum_{i,j}\partial_i\Psi(\bar u)\partial_j\Psi(\bar u)(\partial_t\Gamma_t)(u_i,u_j),
\end{align*}
and
\begin{align*}
 \gamma_{2,t}(\Psi(\bar u))=&\sum_{i,j}\partial_i\Psi(\bar u)\partial_j\Psi(\bar u)\gamma_{2,t}(u_i,u_j)
 +2\sum_{i,j,k}\partial_i\Psi(\bar u)\partial_{jk}\Psi(\bar u)H_t[u_i](u_j,u_k)\\
 +&\sum_{i,j,k,h}\partial_{ik}\Psi(\bar u)\partial_{jh}\Psi(\bar u)\Gamma_t(u_i,u_j)\Gamma_t(u_k,u_h),
\end{align*}
cf. \cite[Lemma 3.3]{savare},
both sides are continuous in $\lambda,a,b$, and hence we conclude that \eqref{poly} holds for all $\lambda, a, b$ in 
$\mathbb R$.

Thus, for $m$-a.e. $x\in X$ we may set $a:=u_2(x)$, $b:=u_3(x)$ so that 
\begin{align*}
 \partial_1\Psi(\bar u)(x)=\lambda,\quad \partial_2\Psi(\bar u)(x)=0=\partial_3\Psi(\bar u)(x)\\
 \partial_{23}\Psi(\bar u)(x)=1=\partial_{32}\Psi(\bar u)(x),\quad \partial_{ij}\Psi(\bar u)(x)=0\text{ else,}
\end{align*}
$m$-a.e., and exploiting \eqref{poly} yields
\begin{align*}
 \lambda^2\gamma_{2,t}(u_1)
 +4\lambda H_t[u_1](u_2,u_3)
 +2\Big(\Gamma_t(u_2,u_3)^2+\Gamma_t(u_2)\Gamma_t(u_3)\Big)\geq\frac12\lambda^2(\partial_t\Gamma_t)(u_1).
\end{align*}
Using Cauchy-Schwartz inequality $\Gamma_t(u_2,u_3)^2\leq \Gamma_t(u_2)\Gamma_t(u_3)$ this can be transformed into
\begin{align*}
 \lambda^2\Big(\gamma_{2,t}(u_1)-\frac12(\partial_t\Gamma_t)(u_1)\Big)
 +4\lambda H_t[u_1](u_2,u_3)
 +4\Gamma_t(u_2)\Gamma_t(u_3)\geq0,
\end{align*}
and since $\lambda$ is arbitrary \cite[Lemma 3.3.6]{gigli2014nonsmooth} we obtain
\begin{align*}
 (H_t[u_1](u_2,u_3))^2\leq\Big(\gamma_{2,t}(u_1)-\frac12(\partial_t\Gamma_t)(u_1)\Big)\Gamma_t(u_2)\Gamma_t(u_3).
\end{align*}
From the definition of the Hessian we deduce that
\begin{align*}
 H_t[u_1](u_2,u_3)+H_t[u_2](u_1,u_3)=\Gamma_t(\Gamma_t(u_1,u_2),u_3)
\end{align*}
and consequently
\begin{equation}
\begin{aligned}\label{u1u2u3}
 |\Gamma_t(\Gamma_t(u_1,u_2),u_3)|
 \leq\sqrt{\Gamma_t(u_3)}\Big(&\sqrt{\gamma_{2,t}(u_1)-\frac12(\partial_t\Gamma_t)(u_1)}\sqrt{\Gamma_t(u_2)}\\
 +&\sqrt{\gamma_{2,t}(u_2)-\frac12(\partial_t\Gamma_t)(u_2)}\sqrt{\Gamma_t(u_1)}\Big).
\end{aligned}
\end{equation}
We obtain \eqref{u1u2u3} for arbitrary $u_3\in\F\cap L^\infty(X,m_t)$ by approximating $u_3$ by a sequence $u_3^n$ converging in energy with
\begin{align*}
\Gamma_t(u_3^n)\to \Gamma_t(u),\qquad \Gamma_t(u_3^n,\Gamma_t(u_1,u_2))\to\Gamma_t(u_3,\Gamma_t(u_1,u_2)) 
\end{align*}
pointwise and in $L^1(X,m_t)$, cf. Theorem 3.4 in \cite{savare}
Hence we may choose $u_3=\Gamma_t(u_1,u_2)$, and obtain the result choosing $u_1=u_2$.

\end{proof}

Now we are ready to prove Theorem \ref{thm1}. 
 
\begin{proof}[Proof of Theorem \ref{thm1}]
  Define for each $\varepsilon>0$ the concave and smooth function $\omega_\varepsilon(\cdot):=(\varepsilon+\cdot)^\alpha-\varepsilon^\alpha$.
 Note that this function satisfies
 \begin{align}\label{omegaest}
2\omega_\varepsilon'(r)+4r\omega_\varepsilon''(r)\geq0.  
 \end{align}

 For each $s,t\in (0,T)$ under consideration as well as $u\in \Lip(X)$ and  $g\in\F\cap L^\infty$ with $g\ge0$, we set
$u_r=P_{r,s}u$, $g_r=P^*_{t,r}g$ for $r\in [s,t]$. Note that for a.e. $r\in[s,t]$ $u_r\in\Dom(\Delta_r)$ and $u,\Gamma_r(u)\in L^\infty(X,m_r)$.

We consider the function 
 \begin{equation*}
  h_r^\varepsilon:=\int g_r\omega_\varepsilon(\Gamma_r(u_r))dm_r.
 \end{equation*}
 Choose $s\le\sigma<\tau\le t$ and $\delta>0$ sufficiently small that $\sigma\leq \tau-\delta$ such that
  \begin{equation}\label{ini-assu}
 h_\tau^\varepsilon\le\liminf_{\delta\searrow 0}\frac1\delta\int_{\tau-\delta}^\tau h_rdr
 \quad\mbox{and}\quad
 h_\sigma^\varepsilon\ge\limsup_{\delta\searrow 0}\frac1\delta\int^{\sigma+\delta}_\sigma h_rdr.
 \end{equation}
Note that by Lebesgue's density theorem, this is true at least for a.e.\  $\sigma\ge s$ and for a.e.\ $\tau\le t$.
 Then from
 \begin{align*}
  \int_{\tau-\sigma}^{\tau}h_r\, dr-\int_\sigma^{\sigma+\delta}h_r\, dr=\int_{\sigma}^{\tau-\delta}(h_{r+\delta}-h_r)\, dr,
 \end{align*}
and the concavity of $\omega_\varepsilon$ we deduce
 \begin{align*}
 h_\tau^\varepsilon-h_\sigma^\varepsilon\le&
 \liminf_{\delta\searrow 0}\frac1\delta\int_\sigma^{\tau-\delta} \big[h_{r+\delta}-h_r\big]dr\\
 \le&
  \limsup_{\delta\searrow 0}\frac1\delta\int_\sigma^{\tau-\delta}\int_X\omega_\varepsilon(\Gamma_{r+\delta}(u_{r+\delta}))d(\mu_{r+\delta}-\mu_r)\,dr\\
  +  &\liminf_{\delta\searrow 0}\frac1\delta\int_\sigma^{\tau-\delta}\int_Xg_r\omega'_\varepsilon(\Gamma_r(u_r))\Big[\Gamma_{r+\delta}(u_{r})-\Gamma_r(u_{r})\Big]dm_r\,dr\\
   +  &\limsup_{\delta\searrow 0}\frac1\delta\int_\sigma^{\tau-\delta}\int_Xg_r\omega_\varepsilon'(\Gamma_r(u_r))
   \Gamma_{r+\delta}(u_{r+\delta},u_{r+\delta}-u_r)\, dm_r\\
   +&\limsup_{\delta\searrow 0}\frac1\delta\int_\sigma^{\tau-\delta}\int_Xg_r\omega_\varepsilon'(\Gamma_r(u_r))\Gamma_{r+\delta}(u_{r+\delta}-u_r,u_r)\, dm_r\,dr\\
   =:& (I) + (II) + (III')+ (III'').
 \end{align*}
 Let us denote with a slight abuse of notation $\hat g_r=g_r\omega_\varepsilon'(\Gamma_r(u_r))$. Note that $\hat g\in L^1\cap L^\infty(X)$ and $\hat g\in\F$.
 Each of the four terms will be considered separately. Since $r\mapsto \mu_r$ is a solution to the dual heat equation, we obtain
  \begin{align*}
  (I)=& \limsup_{\delta\searrow 0}\frac1\delta\int_\sigma^{\tau-\delta}\int_X\omega_\epsilon(\Gamma_{r+\delta} (u_{r+\delta}))\cdot\Big(- \int_r^{r+\delta}\Delta_q g_q\,dm_q\,dq\Big)dr\\
  =& -\liminf_{\delta\searrow 0}\int_{\sigma+\delta}^{\tau}\int_X\omega_\varepsilon(\Gamma_{r} (u_{r}))\Big( \frac{1}\delta\int^r_{r-\delta}\Delta_q g_q e^{-f_q}\,dq\Big)dm_\diamond\,dr\\
  =&-\int_\sigma^\tau\int_X \omega_\varepsilon(\Gamma_r(u_r))\cdot \Delta_rg_r\, dm_r\,dr
   \end{align*}
   due 
   Lebesgue's density theorem applied to $r\mapsto  \Delta_rg_re^{-f_r}$.
   Note that the latter function is in $L^2$ (Theorem \ref{energy-estEva}) and the function $r\mapsto  \omega_\varepsilon(\Gamma_r(u_r))$ is in $L^\infty$ 
   thanks to \eqref{reg-bocheva}.
   
   The second term can estimated according to Proposition \ref{propdiffbar}:
  \begin{align*}
  (II)=&
    \liminf_{\delta\searrow 0}\frac1\delta\int_\sigma^{\tau-\delta}\int_X\hat g_r\Big[ \Gamma_{r+\delta}(u_{r})-\Gamma_r(u_{r})\Big]dm_r\,dr\\
 =& \int_\sigma^\tau\int_X\hat g_r\,(\partial_r\Gamma_r)(u_r)dm_rdr.
    \end{align*}

  The  term $ (III')$ is transformed as follows
   \begin{align*}
  &(III')\\
  =& \limsup_{\delta\searrow 0}\frac1\delta\int_\sigma^{\tau-\delta}\int_X
  \hat g_{r+\delta}\Gamma_{r+\delta}(u_{r+\delta},u_{r+\delta}-u_r)\, dm_{r+\delta}\, dr\\
  =&- \liminf_{\delta\searrow 0}\int^{\tau-\delta}_{\sigma}\int_X\Big(
  \Gamma_{r+\delta}(\hat g_{r+\delta},u_{r+\delta})
  +\hat g_{r+\delta}\,\Delta_{r+\delta}u_{r+\delta}\Big)\Big(\frac1\delta \int_r^{r+\delta}\Delta_qu_q\,dq\Big)dm_{r+\delta}\,dr\\
  =&-\int_\sigma^{\tau}\int_X\Big(
  \Gamma_{r}(\hat g_r,u_{r})+\hat g_r\,\Delta_{r}u_{r}\Big)\cdot \Delta_ru_r\,dm_r\,dr.
           \end{align*}
           Here again we used Lebesgue's density theorem (applied to $r\mapsto  \Delta_ru_r$) and the `nearly continuity' of $r\mapsto \hat g_r$ as map from $(s,t)$ into $L^2(X,m)$ and as map into $\F$  (Lusin's theorem). 
           Moreover, we used  the boundedness (uniformly in $r$ and $x$) of $g_r$ and of $\Gamma_r(u_r)$ as well as the square integrability of $\Delta_ru_r$.

Similarly, the  term $ (III'')$ will be transformed:
   \begin{align*}
  (III'')
  = &\limsup_{\delta\searrow 0}\frac1\delta\int_\sigma^{\tau-\delta}\int_X \hat g_r\Gamma_r(u_{r+\delta}-u_r,u_r)\, dm_r\, dr\\
  =&- \liminf_{\delta\searrow 0}\frac1\delta\int_\sigma^{\tau-\delta}\int_X\Big(
  \Gamma_{r}(\hat g_r,u_{r})+\hat g_r\Delta_{r}u_{r}\Big)\cdot\Big( \int_r^{r+\delta}\Delta_qu_q\,dq\Big)dm_r\,dr\\
  =&-\int_\sigma^{\tau}\int_X\Big(
  \Gamma_{r}(\hat g_r,u_{r})+\hat g_r\,\Delta_{r}u_{r}\Big)\cdot\Big( \Delta_ru_r\Big)dm_r\,dr.
           \end{align*}
  We therefore obtain
  \begin{align*}
 &h_\tau^\varepsilon-h_\sigma^\varepsilon=(I) + (II) + (III')+ (III'')\\
\le& \int_\sigma^\tau\int_X\Big[
- \omega_\varepsilon(\Gamma_r(u_r))\cdot \Delta_rg_r
 +\hat g_r\,(\partial_r\Gamma_r)(u_r)
 -2\big(
  \Gamma_{r}(\hat g_r,u_{r})+\hat g_r\,\Delta_{r}u_{r}\big)\,\Delta_ru_r
 \Big] dm_r\,dr\\
 =&\int_\sigma^\tau\int\Big[\Gamma_r(\Gamma_r(u_r),\hat g_r)-\Gamma_r(\Gamma_r(u_r))\omega_\varepsilon''(\Gamma_r(u_r))g_r
 +\hat g_r\,(\partial_r\Gamma_r)(u_r)\\
 &\qquad\qquad-2\big(\Gamma_{r}(\hat g_r,u_{r})+\hat g_r\,\Delta_{r}u_{r}\big)\,\Delta_ru_r\Big] dm_r\,dr\\
=&\int_\sigma^\tau-2\Gamma_{2,r}(u_r)(\hat g_r)\, dr-\int_\sigma^\tau\int\Big[\Gamma_r(\Gamma_r(u_r))\omega_\varepsilon''(\Gamma_r(u_r))g_r
 +\hat g_r\,(\partial_r\Gamma_r)(u_r)
 \Big] dm_r\,dr.
 \end{align*}
Applying \eqref{mystuff2}, Proposition \ref{gammaestimate}, \eqref{omegaest} and taking into account the concavity of $\omega_\varepsilon$ we further deduce for a.e. $r\in[s,t]$,
 \begin{align*}
 &h_\tau^\varepsilon-h_\sigma^\varepsilon\\
 \leq& \int_\sigma^\tau\int_X\Big[-2\gamma_{2,r}(u_r)\hat g_r+\hat g_r\,(\partial_r\Gamma_r)(u_r)
 -\Gamma_r(\Gamma_r(u_r))\omega_\varepsilon''(\Gamma_r(u_r))g_r\Big]\, dm_r\, dr\\
 \leq& \int_\sigma^\tau\int_X\Big[
 -g_r\Big(\gamma_{2,r}(u_r)-\frac12(\partial_r\Gamma_r)(u_r)\Big)
 \Big(2\omega_\varepsilon'(\Gamma_r(u_r))+4\omega_\varepsilon''(\Gamma_r(u_r))\Gamma_r(u_r)\Big)\Big]\, dm_r\, dr\\
 \leq &0.
   \end{align*}
 
Hence we showed that, given $u$ and $g$, there exists exceptional sets (which are null sets) for $\tau$ and $\sigma$ outside of these sets 
 \begin{equation}\label{fugrad}
 \int_X\omega_\varepsilon(\Gamma_\tau(P_{\tau,\sigma}u)) g\, dm_\tau
 -\int_X P_{\tau,\sigma}\omega_\varepsilon(\Gamma_\sigma(u))\, g\,dm_\tau
  \le  0
     \end{equation}
holds.
Choosing $g$'s from a dense countable set one may achieve that the exceptional sets for $\sigma$ and $\tau$ in \eqref{fugrad} do not depend on $g$.
 Next we may assume that $\sigma,\tau\in [s,t]$ with $\sigma<\tau$ is chosen  such that \eqref{fugrad} simultaneously holds for all $u$ from a dense 
 countable set ${\mathcal C}_1$ in $\Lip(X)$. 
We approximate arbitrary $u \in\Lip(X)$ by $u_n\in{\mathcal C}_1$ in energy and in $L^2$ such that $\sqrt{\Gamma_\tau(P_{\tau,\sigma}u_n)}\rightharpoonup G$
in $L^2$, for some $G\in L^2(X)$. This is possible since $||\sqrt{\Gamma_\tau(P_{\tau,\sigma}u_n)}||_{L^2(X)}$ is uniformly bounded. 
Then we have on the one hand
\begin{align}\label{fugrad1}
 \limsup_{n\to\infty}\int_X P_{\tau,\sigma}\omega_\varepsilon(\Gamma_\sigma(u_n))\, g\,dm_\tau\leq\int_X P_{\tau,\sigma}\omega_\varepsilon(\Gamma_\sigma(u))\, g\,dm_\tau
\end{align}
since
\begin{align*}
&\int_X P_{\tau,\sigma}\omega_\varepsilon(\Gamma_\sigma(u_n))\, g\,dm_\tau-\int_X P_{\tau,\sigma}\omega_\varepsilon(\Gamma_\sigma(u))\, g\,dm_\tau\\
\leq &\int_X P_{\tau,\sigma}^*g\, \omega_\varepsilon'(\Gamma_\sigma(u))(\Gamma_\sigma(u_n)-\Gamma_\sigma(u))\,dm_\sigma\\
\leq &||P_{\tau,\sigma}^*g\, \omega_\varepsilon'(\Gamma_\sigma(u))||_{L^\infty(X)}\left|\int_X\Gamma_\sigma(u_n)-\Gamma_\sigma(u)\,dm_\sigma\right|.
\end{align*}
On the other hand we find 
\begin{align}\label{fugrad2}
 \liminf_{n\to\infty}\int_X\omega_\varepsilon(\Gamma_\tau(P_{\tau,\sigma}u_n)) g\, dm_\tau\geq \int_X\omega_\varepsilon(\Gamma_\tau(P_{\tau,\sigma}u)) g\, dm_\tau.
\end{align}
Indeed, since $P_{\tau,\sigma}u_n\to P_{\tau,\sigma}u$ and $\sqrt{\Gamma(P_{\tau,\sigma}u_n)}\rightharpoonup G$ in $L^2(X)$ we know 
$\Gamma(P_{\tau,\sigma}u)\leq G^2$ $m$-a.e. and hence 
\begin{align*}
 &\int_X\omega_\varepsilon(\Gamma_\tau(P_{\tau,\sigma}u_n)) g\, dm_\tau-\int_X\omega_\varepsilon(\Gamma_\tau(P_{\tau,\sigma}u)) g\, dm_\tau\\
 =&\int_X\tilde\omega_\varepsilon(\sqrt{\Gamma_\tau(P_{\tau,\sigma}u_n)}) g\, dm_\tau-\int_X\tilde\omega_\varepsilon(\sqrt{\Gamma_\tau(P_{\tau,\sigma}u)}) g\, dm_\tau\\
 \geq &\int_X\tilde\omega_\varepsilon'(\sqrt{\Gamma_\tau(P_{\tau,\sigma}u)})(\sqrt{\Gamma_\tau(P_{\tau,\sigma}u_n)}-\sqrt{\Gamma_\tau(P_{\tau,\sigma}u)}) g\, dm_\tau\\
 \geq &\int_X\tilde\omega_\varepsilon'(\sqrt{\Gamma_\tau(P_{\tau,\sigma}u)})(\sqrt{\Gamma_\tau(P_{\tau,\sigma}u_n)}-G) g\, dm_\tau,
\end{align*}
where $\tilde\omega(r)=\omega(r^2)$, which is convex and monotone.
Combining \eqref{fugrad}, \eqref{fugrad1} and \eqref{fugrad2} yields
     \begin{align*}
 &\int_X\omega_\varepsilon(\Gamma_\tau(P_{\tau,\sigma}u)) g\, dm_\tau
 -\int_X P_{\tau,\sigma}\omega_\varepsilon(\Gamma_\sigma(u))\, g\,dm_\tau\\
 \leq&\liminf_n\int_X\omega_\varepsilon(\Gamma_\tau(P_{\tau,\sigma}u_n)) g\, dm_\tau
 -\limsup_n\int_X P_{\tau,\sigma}\omega_\varepsilon(\Gamma_\sigma(u_n))\, g\,dm_\tau\\
 \leq&\liminf_n\left(\int_X\omega_\varepsilon(\Gamma_\tau(P_{\tau,\sigma}u_n)) g\, dm_\tau
 -\int_X P_{\tau,\sigma}\omega_\varepsilon(\Gamma_\sigma(u_n))\, g\,dm_\tau\right)
  \leq 0.
     \end{align*}
 Letting $\varepsilon\to0$ we showed that 
     \begin{align}\label{showneq}
 &\int_X(\Gamma_\tau(P_{\tau,\sigma}u))^\alpha g\, dm_\tau
 \leq\int_X P_{\tau,\sigma}(\Gamma_\sigma(u)^\alpha)\, g\,dm_\tau.
  \end{align}
  Since $\Lip(X)$ is dense in $\F$ we can extend \eqref{showneq} to arbitrary $u\in\F$.
  Since $g$ is arbitrary we obtain the result.
\end{proof}

\subsection{From $L^2$-transport estimates to Bochner's inequality}\label{sec:bochner}
For the proof of Theorem \ref{thm3} it is still left to show that super-Ricci flow implies the pointwise dynamic Bochner inequality \eqref{Bochnereva} at each time $t$. In more detail we show the following.
\begin{thm}\label{thm2}
Let $(X,d_t,m_t)_{t\in I}$ be a one-parameter family of geodesic Polish metric measure spaces satisfying \eqref{assumption1}, \eqref{assumption2}, \eqref{assumption3}
and \eqref{assumption4} such that each $(X,d_t,m_t)$ is a RCD$(K,N)$ space.
If the transport estimate \eqref{transporting} holds, then the pointwise dynamic Bochner inequality \eqref{Bochnereva} holds at all $t\in I$. Moreover the regularity assumption \eqref{reg-bocheva} is satisfied.
\end{thm}

For the proof of Theorem \ref{thm2} we follow the argumentation in the proof of Theorem 5.13 in \cite{sturm2016}. The argumentation in \cite{sturm2016} is inspired
by \cite{bggk2015}, where the authors prove the equivalence between Wasserstein contraction estimates and Bochner's inequality in the static setting.

\begin{proof}[Proof of Theorem \ref{thm2}]
Since the transport estimate \eqref{transporting} implies the $L^2$-gradient estimate and the heat flow satisfies the maximum principle,
the regularity assumption is clearly satisfied.

Define $u=h_\varepsilon^tu_0$, where $u_0\in L^\infty(X,m_t)\cap L^2(X,m_t)$ and $h_\varepsilon^t$ the static semigroup mollification
\begin{align*}
 h_\varepsilon^tu_0:=-\frac1{\varepsilon^2}\int_0^\infty H_r^tu_0\kappa(\frac{r}{\varepsilon})\, dr.
\end{align*}
Here, $(H_r^t)_{r\geq0}$ denotes the (static) semigroup associated to $\E_t$ and $\kappa\in\mathcal C_c^\infty((0,\infty))$ with $\kappa\geq0$ and 
$\int_0^\infty\kappa_r\, dr=1$. Recall that $u,\Delta_tu\in\Dom(\Delta_t)\cap\Lip_b(X)$.

Let $g\in\F\cap L^\infty(X,m_t)$ such that $g\geq0$. Then, the transport estimate \eqref{transporting} together with Lemma 5.10 and Lemma 5.11 in 
\cite{sturm2016} eventually yields
 \begin{align*}
  &-\frac{1}2 \int P_{t,s}(\Gamma_s(u)) gdm_t+\int \Gamma_t(P_{t,s}u,u)gdm_t\leq\frac{1}{2}\int\Gamma_t(u)gdm_t,
 \end{align*}
see \cite[Section 5]{sturm2016}. Then following the lines in \cite{sturm2016},
we subtract $\frac12\int\Gamma_t(u)gdm_t$ on each side and divide by $t-s$ obtaining
 \begin{equation}
 \begin{aligned}\label{decomposition}
  &\frac{1}{2(t-s)}\left[\int\Gamma_t(u)gdm_t-\int P_{t,s}(\Gamma_s(u)) gdm_t\right]\\
  &+\frac1{t-s}\left[\int \Gamma_t(P_{t,s}u,u)gdm_t-\int\Gamma_t(u)gdm_t\right]\\
  &\leq 0.
 \end{aligned}
 \end{equation}
 We decompose the first term on the left-hand side into the following two terms
 \begin{align*}
  &\frac1{2(t-s)}\left[\int\Gamma_t(u)gdm_t-\int\Gamma_s(u)P^*_{t,s}g dm_s\right]\\
  =&\frac1{2(t-s)}\left[\int\Gamma_t(u)gdm_t-\int\Gamma_t(u)P^*_{t,s}g dm_s\right]+\frac1{2}\int\frac{\Gamma_t(u)-\Gamma_s(u)}{t-s}P^*_{t,s}g dm_s.\\
 \end{align*}
 Recall that $\Gamma_t(u)\in\F$ \cite[Lemma 3.2]{savare} and thus we can apply Lemma \ref{P*Eva}, which gives us
 \begin{align}\label{decomp1.1}
 \lim_{s\nearrow t} \frac1{(t-s)}\left[\int\Gamma_t(u)gdm_t-\int\Gamma_t(u)P^*_{t,s}g dm_s\right]=\int\Gamma_t(\Gamma_t(u),g)dm_t,
 \end{align}
while, since $|\frac{\Gamma_s(u)-\Gamma_t(u)}{(t-s)}|\leq2L\Gamma_t(u)\in L^\infty(X,m_t)$,
\begin{equation}
\begin{aligned}\label{decomp1.2}
 &\liminf_{s\nearrow t}\int\frac{\Gamma_t(u)-\Gamma_s(u)}{(t-s)}(P^*_{t,s}g)dm_s\\
 &\geq \liminf_{s\nearrow t}\int\frac{\Gamma_t(u)-\Gamma_s(u)}{(t-s)}gdm_t+\liminf_{s\nearrow t}\int\frac{\Gamma_t(u)-\Gamma_s(u)}{(t-s)}(P^*_{t,s}ge^{-f_s}-ge^{-f_t})dm\\
 &\geq\int(\partial_t\Gamma_t)(u)gdm_t-\limsup_{s\nearrow t}2L||\Gamma_t(u)||_{L^\infty(X,m_t)}||P^*_{t,s}ge^{-f_s}-ge^{-f_t}||_{L^1(X,m_t)}\\
 &=\int(\partial_t\Gamma_t)(u)gdm_t,
\end{aligned}
\end{equation}
where we used Proposition \ref{propdiffbar} in the last inequality and that $P^*_{t,s}ge^{-f_s}\to ge^{-f_t}$ in $L^1(X,m)$ as $s\to t$.

 Regarding the second term on the left-hand side of \eqref{decomposition}, 
 note that the Leibniz rule and the integration by parts formula is applicable and we get
\begin{equation}\label{eq:chain}
\begin{aligned}
 \int \Gamma_t(P_{t,s}u,u)gdm_t=\int\Gamma_t(gP_{t,s}u,u)dm_t-\int\Gamma_t(g,u)P_{t,s}u dm_t\\
 =-\int \psi P^*_{t,s}(g\Delta_tu)dm_s-\int P^*_{t,s}(\Gamma_t(g,u))u dm_s.
\end{aligned}
\end{equation}
Subtracting $\int\Gamma_t(u)gdm_t$ and applying \eqref{eq:chain}
\begin{align*}
 &\frac1{(t-s)}(\int \Gamma_t(P_{t,s}u,u)gdm_t-\int\Gamma_t(u)gdm_t)\\
 =&\frac1{(t-s)}(-\int \psi P^*_{t,s}(g\Delta_tu)dm_s+\int\psi(g\Delta_tu)dm_t)\\
 +&\frac1{(t-s)}(-\int P^*_{t,s}(\Gamma_t(g,u))u dm_s+\int\Gamma_t(u,g)u dm_t).
\end{align*}
Letting $s\nearrow t$ we have since $g\in\F\cap L^\infty(X,m_t)$ and $\Delta_tu\in \Lip_b(X)$, $g\Delta_tu\in \F\cap L^1(X,m_t)$ 
\begin{align*}
 \lim_{s\nearrow t}\frac1{(t-s)}(-\int u P^*_{t,s}(g\Delta_tu)dm_s+\int u(g\Delta_tu)dm_t)=
 \int\Gamma_t(u,g\Delta_tu)dm_t
\end{align*}
by virtue of Lemma \ref{P*Eva}. In order to determine
\begin{align*}
 \lim_{s\nearrow t}\frac1{(t-s)}(-\int P^*_{t,s}(\Gamma_t(g,u))u dm_s+\int\Gamma_t(u,g)u dm_t),
\end{align*}
we need to argue whether $\Gamma_t(g,u)\in\F$. But this is the case, since, due to our static RCD$(K,\infty)$ assumption, we may apply 
Theorem 3.4 in \cite{savare} and obtain
\begin{align*}
 \Gamma_t(\Gamma_t(g,u))\leq 2(\gamma_2(u)-K\Gamma_t(u))\Gamma_t(g)+2(\gamma_2(g)-K\Gamma_t(g))\Gamma_\tau(u)\quad m_t\text{-a.e.},
\end{align*}
where $\gamma_2(u),\gamma_2(g)\in L^1(X,m_t)$. Our regularity assumptions on $u$ and $g$ provide that the right hand side 
is in $L^1(X,m_t)$ and consequently Lemma \ref{P*Eva} implies
\begin{align*}
 \lim_{s\nearrow t}\frac1{(t-s)}(-\int P^*_{t,s}(\Gamma_t(g,u))u dm_s+\int\Gamma_t(u,g)u dm_t)
 =\int\Gamma_t(\Gamma_t(g,u),u) dm_t.
\end{align*}
Combining these observations we find
\begin{equation}
\begin{aligned}\label{decomp2}
 &\lim_{s\nearrow t}\frac1{(t-s)}(\int \Gamma_t(P_{t,s}u,u)gdm_t-\int\Gamma_t(u)gdm_t)\\
 =&\int \Gamma_t(u,g\Delta_tu)dm_t+\int\Gamma_t(\Gamma_t(g,u),u) dm_t
 =-\int(\Delta_tu)^2g+\Gamma_t(g,u)\Delta_tudm_t.
\end{aligned}
\end{equation}
Hence from \eqref{decomposition}, \eqref{decomp1.1}, \eqref{decomp1.2} and \eqref{decomp2}
\begin{align*}
 \frac12\int(\partial_t\Gamma_t)(u)gdm_t+ \frac12\int\Gamma_t(\Gamma_t(u),g)dm_t
 \leq\int(\Delta_tu)^2g+\Gamma_t(g,u)\Delta_tudm_t.
\end{align*}

Let now $g\in\Dom(\Delta_t)\cap L^\infty(X,m)$ with $g\geq0$ and $u\in\Dom(\Delta_t)\cap L^\infty(X,m_t)$ with $\Gamma_tu\in L^\infty(X,m_t)$.
Then from the above argumentation we obtain
\begin{align*}
&\frac12\int\Gamma_t(h_\varepsilon^tu)\Delta_tg\, dm_t+\int(\Delta_t(h_\varepsilon^tu))^2g+\Gamma_t(g,h_\varepsilon^tu)\Delta_t(h_\varepsilon^tu)dm_t\\
&\geq  \frac12\int(\partial_t\Gamma_t)(h_\varepsilon^tu_n)gdm_t.
\end{align*}
Since $(\partial_t\Gamma_t)(u)(x)=-2H_t(x)\Gamma_t(u)(x)$ for $m$-a.e. $x\in X$ and $|H_t(x)|\leq C$, we obtain the assertion by letting $\varepsilon\to 0$
with taking into account that 
\begin{align*}
 ||h_\varepsilon^t u-u||_\F\to0\text{ as }\varepsilon\to0\text{ and }\Delta_th_\varepsilon^t u=h_\varepsilon^t\Delta_t u.
\end{align*}

\end{proof}

\subsection{Proof of the equivalence result}\label{sec:proofofeq}
It remains to finalize the proof of Theorem \ref{thm3}. Before we do so let us quickly prove the following Lemma.

\begin{lma}\label{cor:kuwa}
Suppose that $(X,d_t,m_t)_{t\in I}$ is a super-Ricci flow satisfying \eqref{assumption3} and \eqref{assumption4}. Then
\begin{enumerate}
\item[i)] for every $u\in \F\cap L^\infty(X,m_s)$ and every $\beta\in[1,2]$
\begin{align}\label{eq:kuwa}
|\nabla_tP_{t,s}u|_*^\beta\leq P_{t,s}(|\nabla_s u|_*^\beta),
\end{align}
\item[ii)] for every $\mu,\nu\in\mathcal P(X)$ and every $p\in[1,\infty]$
\begin{align}\label{eq:kuwa23}
W_{p,s}(\hat P_{t,s}\mu,\hat P_{t,s}\nu)\leq W_{p,t}(\mu,\nu).
\end{align}
\end{enumerate}

\end{lma}
\begin{proof}
Note that, taking into account $\Gamma(u)=|\nabla u|_*^2$ 
due to our static Riemannian curvature bound, \eqref{eq:kuwa} holds at least for a.e. $s\leq t$ by Definition \ref{supereva}, Theorem \ref{thm2} and 
Theorem \ref{thm1}.
Then applying Kuwada's duality \cite[Theorem 2.2]{kuwadadual} implies that \eqref{eq:kuwa2} holds at all these time instances. 
Indeed, \eqref{eq:kuwa} implies that for all $u\in\Lip_b(X)$,
$|\nabla_tP_{t,s}u|_*\leq P_{t,s}(|\nabla_s u|_*^\beta)^{1/\beta}$ and thus by Proposition 3.11 in \cite{agsbe}
$\lip_tP_{t,s}u\leq P_{t,s}(|\nabla_s u|_*^\beta)^{1/\beta}$. We obtain 
\begin{align}\label{eq:kuwab}
 \lip_t(P_{t,s}u)\leq P_{t,s}(\lip_s(u)^\beta)^{1/\beta}
\end{align}
by virtue of $|\nabla u|_*= \lip\, u$ (Theorem 6.1 in \cite{cheeger} and Theorem 6.2 in \cite{agscalc}).
We deduce from Theorem 2.2 in \cite{kuwadadual} for a.e. $s\leq t$
\begin{align*}
 W_{p,s}(\hat P_{t,s}\mu,\hat P_{t,s}\nu)\leq W_{p,t}(\mu,\nu),
\end{align*}
where $p$ is the H\"older conjugate of $\beta$; $1/p+1/\beta=1$. Since both sides of the above equation are continuous 
in $s$ and $t$ 
(see Lemma \ref{continuitylp}), we obtain that it holds for all times $s\leq t$ and thus also 
\eqref{eq:kuwab} holds for all times by Theorem 2.2 in \cite{kuwadadual}. Following the same argumentation as above we find that 
\eqref{eq:kuwa} holds for all $s\leq t$ and all $u\in \Lip_b(X)$. By density of $\Lip(X)\subset \F$ we find $i)$.
The same applies to $p=1$ in \eqref{eq:kuwa23} by noting that 
$\lip_t(P_{t,s}u)\leq P_{t,s}(\lip_s(u)^\beta)^{1/\beta}$ for all $\beta\geq 2$ by virtue of Jensen's inequality.
\end{proof}
\begin{proof}[Proof of Theorem \ref{thm3}]
By virtue of Lemma \ref{cor:kuwa}, Theorem \ref{brownian} and Definition \ref{supereva} the only implication left to show is $iii)$ implies one of the other assertions. We show that $iii)$ implies $ii)$. This can be seen by choosing an $W_{p,t}$-optimal transport plan $\gamma$ between $\mu,\nu$. Let $q_{t,s}$ be the transition kernel of the coupled process. Then $q_{t,s}(z_1,dx,z_2,dy)\gamma(dz_1,dz_2)$ is a coupling of $\hat P_{t,s}\mu$ and $\hat P_{t,s}\nu$. Hence
\begin{align*}
&W_{p,s}(\hat P_{t,s}\mu,\hat P_{t,s}\nu)^p\leq\int \int d_s^{p}(x,y)q_{t,s}(z_1,dx,z_2,dy)\gamma(dz_1,dz_2)\\
=&\int E[d^p_s(X^1_s,X^2_s)|X^1_t=z_1,X^2_t=z_2] \gamma(dz_1,dz_2)\leq \int d^p_t(z_1,z_2)\gamma(dz_1,dz_2)\\
=&W_{p,t}(\mu,\nu)^p,
\end{align*}
where we used $iii)$ in the last inequality.
\end{proof}
\newpage
\bibliography{srf-KoSt-6}

\end{document}